\documentclass[12pt]{article}

\usepackage{amsbsy,amsthm}
\usepackage{amsfonts}
\usepackage{amsmath}
\usepackage{amssymb}
\usepackage{amsthm}
\usepackage{bbm}
\usepackage{booktabs}
\usepackage[font=small,labelfont=bf,labelsep=quad]{caption}
\usepackage{cite}
\usepackage{color}
\usepackage{colortbl,dcolumn}
\usepackage{enumerate}
\usepackage{float}
\usepackage[top=0.8in, bottom=0.8in, left=0.8in, right=0.8in]{geometry}
\usepackage{graphicx}
\usepackage{indentfirst}
\usepackage{latexsym}
\usepackage{mathrsfs}
\usepackage[section]{placeins}
\usepackage{psfrag}
\usepackage{verbatim}
\usepackage{xcolor}

\numberwithin{equation}{section}

\providecommand{\keywords}[1]{\textbf{\text{Keywords: }} #1}

\def\R{\mathbb{R}}
\def\E{\mathbb{E}}
\def\N{\mathbb{N}}
\def \e{\text{e}}
\def \half{\tfrac12}
\def \d{{\rm d}}

\newcommand{\dd}{\text{d}}

\numberwithin{equation}{section}

\newtheorem{thm}{Theorem}[section]

\newtheorem{lem}[thm]{Lemma}

\newtheorem{cor}[thm]{Corollary}

\newtheorem{example}[thm]{Example}
\newtheorem{assumption}[thm]{Assumption}

\begin{document}
	
\title{Unconditionally positivity-preserving explicit Euler-type schemes for a generalized A\"{i}t-Sahalia model
\thanks {This work was supported by Natural Science Foundation of China (12071488, 12371417, 11971488) and Natural Science Foundation of Hunan Province (2020JJ2040).}}
\author{\normalsize Ruishu Liu,\quad Yulin Cao, \quad Xiaojie Wang{\thanks{Corresponding author: x.j.wang7@csu.edu.cn}}\\
\footnotesize School of Mathematics and Statistics, HNP-LAMA, Central South University, Changsha, China}

\maketitle

\begin{abstract}
The present work is devoted to strong approximations of a generalized A\"{i}t-Sahalia model arising from mathematical finance. 
The numerical study of the considered model faces essential difficulties caused by a drift that blows up at the origin, highly nonlinear drift and diffusion coefficients and positivity-preserving requirement.
In this paper, a novel explicit Euler-type scheme is proposed, which is easily implementable and  able to preserve positivity of the original model unconditionally, i.e., for any time step-size $h >0$.
A mean-square convergence rate of order $0.5$ is also obtained for the proposed scheme in both non-critical and general critical cases.
Our work is motivated by the need to justify the multi-level Monte Carlo (MLMC) simulations for the underlying model, where the rate of mean-square convergence is required and the preservation of positivity is desirable particularly for large discretization time steps.
Numerical experiments are finally provided to confirm the theoretical findings.
\end{abstract}

\keywords
    A\"{i}t-Sahalia model; 
    Explicit scheme; 
    Unconditionally positivity preserving;
    Strong convergence rate.

\section{Introduction}

The pricing of financial assets has been an important issue in mathematical finance over the past few decades, where
stochastic differential equations (SDEs) are prominent models frequently used to capture the behaviors of financial systems. 
Notable examples include the Black-Scholes model \cite{black1973pricing}, the Merton model \cite{merton1973theory}, and the Cox-Ingersoll-Ross (CIR) model \cite{cox2005theory} in the modern financial theory.
The economist and mathematician 
A\"{i}t-Sahalia conducted an empirical study in \cite{ait1996testing} and investigated several continuous-time models for interest rates. 
He tested the models by comparing the parameter density implied by the parameter models and found that all the existing univariate linear drift models fail to explain the euro dollar well.
As a consequence of this work, a nonlinear stochastic interest model which captures well of the dynamics of the spot rate was proposed and named as the A\"{i}t-Sahalia model, which is now widely used in finance and economics, 
and described by the following It\^{o}-type SDE:
\begin{equation}\label{2023AS-introduction_AS_model_kappa=2}
    \mathrm{d} X_t
    =
    \big( c_{-1} X_t^{-1} - c_0 + c_1 X_t - c_2 X_t^{2} \big)
    \, \dd t
    +
    c_3 X_t^{\rho} 
    \, \dd W_t,
    \quad
    X_0=x_0 > 0.
\end{equation}
Here, $c_{-1}, c_{0}, c_{1}, c_{2}, c_{3}> 0$, $\rho>1$ and $\{ W_t \}_{t \in [0, T]}$ stands for a standard Brownian motion defined on a complete probability space $(\Omega, \mathcal{F}, \mathbb{P})$.  
A\"{i}t-Sahalia \cite{ait1996testing} proved that the model \eqref{2023AS-introduction_AS_model_kappa=2} admits a unique global positive solution using the Feller explosion test.
In 2009, Cheng \cite{cheng2009highly} derived the finite moment bounds for the analytic solutions and proved that the Euler–Maruyama approximate solutions are convergent in probability.
Later in 2011, Szpruch et al. \cite{szpruch2011numerical} extended the model \eqref{2023AS-introduction_AS_model_kappa=2} to a more general version given by
\begin{equation}\label{2023AS-introduction_AS_model_kappa}
\begin{aligned}
    \mathrm{d} X_t
    = & 
    (c_{-1} X_t^{-1} - c_0 + c_1 X_t - c_2 X_t^{\kappa})
    \, \dd t
    +
    c_3 X_t^{\rho} \dd W_t,
    \quad
    X_0 = x_0 >0,
    \:
    \kappa > 1,
\end{aligned}
\end{equation}	
and showed that the generalized model 
admits a unique global positive solution with
uniformly in time bounded $p$-th moments ($ p \in \R$)
in the non-critical case $\kappa + 1 > 2 \rho$. 
Moreover, they applied the classical backward Euler-Maruyama (BEM) scheme 
to approximate the model and proved the 
positivity-preserving approximations to be strongly convergent under the condition that $\kappa+1>2\rho$, but with no convergence rates revealed.
Recently, the authors of \cite{wang2020mean} successfully recovered the mean-square convergence rate of order $1/2$ for the stochastic theta methods (STMs) with $\theta \in [\half,1]$ applied to the model \eqref{2023AS-introduction_AS_model_kappa} under the condition that $\kappa+1 \geq 2\rho$, covering the critical case $\kappa+1 = 2\rho$. 
Also, we mention that Neuenkirch and Szpruch \cite{neuenkirch2014first}
proposed a kind of Lamperti-backward Euler method for scalar SDEs defined in a domain, including the Ait-Sahalia-type model \eqref{2023AS-introduction_AS_model_kappa}, with a strong
convergence rate of order one obtained.
However, the computational costs of 
solving implicit  algebraic equations 
rise as the parameter $\kappa$ increases. 
More recently, 
a truncated explicit Euler method was proposed in \cite{emmanuel2021truncated} to approximate 
\eqref{2023AS-introduction_AS_model_kappa}, but with no convergence rate revealed.
When we finish the draft of the present work,
we are also aware of two closely related publications \cite{deng2023positivity,Lord2024convergence}. 
The authors of \cite{deng2023positivity} 
recovered a mean-square 
convergence rate of order nearly $\tfrac14$ for 
the truncated Euler method only in the non-critical case $\kappa+1>2\rho$.
The work \cite{Lord2024convergence} proposed
an exponential tamed method
based on the Lamperti transform and splitting,
which is strongly convergent with order one,
but not strictly positivity preserving.
{\color{black}
Finally, it is worthwhile to mention an interesting paper \cite{halidias2022boundary}, where, by combining a Lamperti-type transformation with a semi-implicit approach, the authors proposed a positivity preserving scheme for the model in the particular critical case $\kappa=2, \rho = \tfrac32$. For the particular critical case, the proposed scheme 
can be explicitly solved, by finding a positive root of a quadratic equation. Moreover, 
it was shown to have order one strong convergence (see \cite{halidias2022boundary} for more details). 
}

In the present article, we aim to design novel 
explicit positivity preserving schemes for
approximations of the A\"{i}t-Sahalia model 
\eqref{2023AS-introduction_AS_model_kappa} with a mean-square convergence rate of order $\tfrac12$.
More precisely,
on a uniform mesh with a uniform 
step-size $h = \tfrac{T}{N},\, N \in \mathbb{N}$
we propose the following time-stepping scheme:
\begin{equation}
    Y_{n+1}
    =
    Y_{n}
    +
    c_{-1} Y_{n+1}^{-1} h
    +
    \big( - c_0 + c_1 Y_n + f_h(Y_n) \big) h
    +
    g_h(Y_n) \Delta W_n,
    \quad
    Y_0 = X_0,
\end{equation}
where $f_h$ and $g_h$ are 
some explicit modifications of $f (x) := - c_2 x^{\kappa}$ and $g (x) : = c_3 x ^{\rho}$ in the SDE model, respectively (see Sect. \ref{2023AS-section:EPE_scheme} for details).
Clearly, the implicitness only appears in the term
$c_{-1} Y_{n+1}^{-1} h $, which is the key element
 to preserve the positivity of the original model 
 (see Lemma \ref{lem:positive-preserving}).
On the contrary,  the remaining terms are explicitly handled. It is not difficult to see
that, given the former step $Y_n$, 
the current step $Y_{n+1}$ can be explicitly 
known, as the unique positive root of a quadratic equation (see \eqref{eq:scheme-solution} below for 
the explicit solution).
In this way, the computational costs are significantly reduced compared with the implicit schemes (BEM or STMs)
in existing publications \cite{szpruch2011numerical,wang2020mean}, 
where implicit algebraic equations are required
to be numerically solved for every step.
Furthermore, the proposed explicit Euler-type 
scheme  is  positivity-preserving for any time step-size 
$h > 0$, which is  particularly desirable for the multi-level Monte Carlo (MLMC) simulations \cite{giles08multilevel} with large discretization time steps. 

To justify the MLMC simulations, the mean-square convergence rate of the approximations is also necessary
\cite[Theorem 3.1]{giles08multilevel}.
Recall that the analysis of convergence rates in  \cite{wang2020mean} for the implicit schemes does not rely on a priori moment estimates of numerical approximations.
Nevertheless,  due to the presence of the explicit treatment
in the drift term, the analysis of convergence rates
for the proposed semi-implicit schemes relies on 
a finite moment bound of numerical approximations:
%
\begin{equation} \label{eq:intro-moment-bounds}
    \sup_{ N \in \N}
    \sup _{0 \leq n \leq N} 
    \E \big[ \vert Y_{n} \vert ^{p} \big]
    < \infty,
    \quad
    p \geq 2.
\end{equation}
%
Deriving such $p$-th moment integrability of the numerical approximations is, however, not an easy task
and some non-standard arguments are exploited (see the proof of Lemma \ref{2023AS-lem:Y_integrability_gamma}) to overcome essential difficulties caused by a drift that blows up at the origin, highly nonlinear diffusion coefficients and a mixture of implicitness and explicitness
in the drift part of the scheme.
%
%
Equipped with the moment bound \eqref{eq:intro-moment-bounds}, one is able to identify 
the mean-square convergence rate of order $0.5$ 
for the proposed scheme \eqref{2023AS-eq:EUPE_scheme} 
applied to the model \eqref{2023AS-introduction_AS_model_kappa}
(see Theorem \ref{2023AS-thm:convergence_rate_k+1>=2rho}):
\begin{equation}
    \sup _{0\leq n\leq N}
    \E \big[
        \big\vert X_{t_n}-Y_{n}\big\vert^2
    \big]
    \leq 
    C h
\end{equation}
in both non-critical ($\kappa + 1 > 2 \rho$) and 
general critical ($\kappa + 1 = 2 \rho$) cases,
where the constant $C$ is independent of the time step-size $h$.
To the best of our knowledge, 
    this is the first paper to devise 
    an unconditionally 
    positivity preserving explicit scheme
    with order $\tfrac12$ of mean-square convergence
    {\color{black}for the non-critical case of 
    the  A\"{i}t-Sahalia model \eqref{2023AS-introduction_AS_model_kappa}.}
%

The rest of this paper is organized as follows.
In the next section, basic notations and properties of the original model are presented. 
In Sect. \ref{2023AS-section:EPE_scheme} the newly proposed scheme is introduced and proved to be positivity preserving and moment bounded.
The mean-square error bounds for the proposed scheme is 
carefully analyzed in Sect. \ref{2023AS-section:convergence_rate}, with order one-half of convergence revealed.
Numerical experiments are shown in Sect. \ref{2023AS-section:numerical_experiments} to verify our theoretical results, and a short conclusion is made in Sect. \ref{2023AS-section:conclusion}.

\section{The generalized A\"{i}t-Sahalia model}
Throughout this paper, we use the following notations.
Let $ \vert  \cdot \vert  $ and $ \langle \cdot ,\cdot \rangle $
be the Euclidean norm and the inner product in $ \mathbb{R} $, respectively.
Let $\{ W_t \}_{t \geq 0}$ stand for a standard Brownian motion defined on a complete filtered probability space $(\Omega, \mathcal{F}, (\mathcal{F}_t)_{t\geq 0}, \mathbb{P})$, where the filtration satisfies usual conditions, and let $ \E $ denote the expectation.
Let $ L^p(\Omega;\R) $ ($p \geq 1$) denote the space consisting of all $ \R $-valued $ p $-times integrable random variables with the norm $\Vert \cdot \Vert_{L^p(\Omega;\R)} $ defined as
\begin{equation}
    \Vert  \xi \Vert _{L^p(\Omega;\R)} 
    := 
    \big ( \E \big[ \vert \xi \vert ^p \big] \big)^{ 1/p} < \infty.
\end{equation}
%
In this work, we focus on the generalized A\"{i}t-Sahalia model taking the form as follows:
\begin{equation}\label{2023AS-eq:model_SDE}
	\dd X_t= (c_{-1} X_t^{-1} - c_{0}+ c_{1} X_t - c_{2} X_t^{\kappa})
	\, \dd t
	+c_{3} X_t^{\rho} \dd W_t , \quad
	X_0=x_0, 
	\quad
	t > 0,
\end{equation}
where $c_{-1},c_{0},c_{1},c_{2},c_{3}>0$, $\kappa$, $\rho>1$ and $\kappa+1\geq2\rho$.
For the particular case $\kappa = 2$, the underlying model \eqref{2023AS-eq:model_SDE} reduces into \eqref{2023AS-introduction_AS_model_kappa=2}, originally proposed by
A\"{i}t-Sahalia \cite{ait1996testing}.
%
%
The well-posedness of the generalized model \eqref{2023AS-eq:model_SDE} has been established in the following theorem quoted from \cite[Theorem 2.1]{szpruch2011numerical}.
\begin{thm}
\label{thm:AS-model-well-posedness}
Let $c_{-1},c_{0},c_{1},c_{2},c_{3}>0$, $\kappa$, $\rho>1$.
    Given any initial value $X_0 > 0$, there exists a unique, positive global solution $X_t$ to the SDE \eqref{2023AS-eq:model_SDE} on $t \geq 0$.
\end{thm}

Furthermore, we recall the moment integrability properties
of the solution to \eqref{2023AS-eq:model_SDE} as follows
(see \cite[Lemma 2.1]{szpruch2011numerical} and 
\cite[Lemma 4.6]{wang2020mean} for details).
\begin{lem}
\label{2023AS-lem:X_integrability_k+1>2rho}
    Let conditions in Theorem \ref{thm:AS-model-well-posedness} be satisfied with $\kappa + 1 > 2 \rho$. For any $p \geq 2$, it holds
    \begin{equation}
        \sup_{0 \leq t < \infty} 
        \E \big[ \vert X_t \vert^p \big]
        <
         \infty
        \quad
        \text{and}
        \quad
        \sup_{0 \leq t < \infty} 
        \E \big[ \big\vert \tfrac{1}{X_t} \big\vert^p \big]
        <
         \infty.
    \end{equation}
\end{lem}

For the general critical case $\kappa+1 = 2 \rho$, 
we have slightly different assertions, as stated 
in the following lemma (cf. \cite[Lemma 4.6]{wang2020mean}).
\begin{lem}
\label{2023AS-lem:X_integrability_k+1=2rho}
    Let conditions in Theorem \ref{thm:AS-model-well-posedness} be satisfied with $\kappa + 1 = 2 \rho$.
    Then for any $2 \leq p_1 \leq \tfrac{2 c_2}{c_3^2} + 1$ and for any $p_2 \geq 2$, we have
    \begin{equation}
        \sup_{0 \leq t < \infty} 
        \E \big[ \vert X_t \vert^{p_1} \big]
        <
        \infty
        \quad
        \text{and}
        \quad
        \sup_{0 \leq t < \infty} 
        \E \big[ \big\vert \tfrac{1}{X_t} \big\vert^{p_2} \big]
        <
        \infty.
    \end{equation}
\end{lem}
Based on the above moment integrability properties, 
one can also obtain the H\"{o}lder continuity of the solutions, as stated in the following lemmas 
quoted directly from \cite{wang2020mean}.
\begin{lem}
\cite[Lemma 4.4]{wang2020mean}
\label{2023AS-lem:Holder_continuous_k+1>2rho}
    Let conditions in Theorem \ref{thm:AS-model-well-posedness} be satisfied with $\kappa + 1 > 2 \rho$.
    Then it holds that, for any $p \geq 1$ and $t,s \in[0,T]$,
    \begin{align}
        \Vert X_t - X_s \Vert_{L^{p}(\Omega; \R)}
        & \leq 
        C \vert t-s \vert^{\half},\\
        \Vert X_t^{-1} - X_s^{-1} \Vert_{L^{p}(\Omega; \R)}
        & \leq 
        C \vert t-s \vert^{\half}.
    \end{align}
\end{lem}

\begin{lem}\cite[Lemma 4.7]{wang2020mean}
\label{2023AS-lem:Holder_continuous_k+1=2rho}
    Let conditions in Theorem \ref{thm:AS-model-well-posedness} be satisfied with $\kappa + 1 = 2 \rho$.
    Then for any $t,s \in[0,T]$, it holds that
    \begin{align}
        \Vert X_t - X_s \Vert_{L^{q_1}(\Omega; \R)}
        & \leq 
        C \vert t-s \vert^{\half},
        \quad
        2 
        \leq 
        q_1 
        \leq 
        \tfrac{1}{\kappa} 
        \big( \tfrac{2 c_2}{c_3^2} + 1 \big), \\
        \Vert X_t^{-1} - X_s^{-1} \Vert_{L^{q_2}(\Omega; \R)}
        & \leq 
        C \vert t-s \vert^{\half},
        \quad
        2 
        \leq 
        q_2
        <
        \tfrac{1}{\kappa} 
        \big( \tfrac{2 c_2}{c_3^2} + 1 \big).
    \end{align}
\end{lem}

Here and below, for notation simplicity, 
the letter $ C $ is slightly abused to denote a 
generic positive constant, which is independent 
of the time step-size and may vary for each appearance.

As a straightforward consequence of 
the above lemmas, one can use the H\"{o}lder inequality to show the following estimates
(consult \cite[(4.21), (4.24)]{wang2020mean}).
\begin{lem}\label{2023AS-lem:f_L^2_bounds}
Let conditions in Theorem \ref{thm:AS-model-well-posedness} be satisfied with $\kappa + 1 \geq 2 \rho$. If one of the following conditions stands:
\begin{enumerate}[(1)]
    \item $\kappa + 1 > 2 \rho$,
    \item $\kappa + 1 = 2 \rho$, $\tfrac{c_2}{c_3^2} > 2\kappa - \tfrac32$,
\end{enumerate}
then for any $t,s \in[0,T]$, it holds that
\begin{equation}
\| f (X_t)- f (X_s) \|_{L^{2}(\Omega; \R)}
\leq
C |t - s|^{\frac12},
\quad
\| g (X_t)- g (X_s) \|_{L^{2}(\Omega; \R)}
\leq
C |t - s|^{\frac12},
\end{equation}
where we denote $f (x) := - c_2 x^{\kappa}$ and $g (x) : = c_3 x ^{\rho}$ for short.
\end{lem}

Armed with the above properties of the underlying model,
one can turn to the numerical approximations and
analyze the resulting mean-square errors in the forthcoming
sections.

\section{The explicit positivity-preserving Euler-type scheme} \label{2023AS-section:EPE_scheme}

In this section, we aim to propose and analyze new
strong approximation schemes for the generalized A\"{i}t-Sahalia model \eqref{2023AS-eq:model_SDE}. 
For $T \in (0, +\infty)$ and $N \in \N$, we construct 
a uniform mesh $\{t_n = n h\}_{n = 0}^N$ on the time interval $[0, T]$ with a uniform step-size $ h = \tfrac{T}{N}$.
Based on the uniform mesh, we propose an explicit Euler-type method for \eqref{2023AS-eq:model_SDE} as follows:
\begin{equation}\label{2023AS-eq:EUPE_scheme}
    Y_{n+1}
    =
    Y_{n}
    +
    c_{-1} Y_{n+1}^{-1} h
    +
    \big( - c_0 + c_1 Y_n + f_h(Y_n) \big) h
    +
    g_h(Y_n) \Delta W_n,
    \quad
    Y_0 = X_0,
\end{equation}
where
$
    \Delta W_{n} := W_{t_{n+1}} - W_{t_{n}}
$
for $n = 0,1,...,N-1$
and
{\color{black}
$f_h,g_h: (0, + \infty) \rightarrow \R$ are some 
modifications of 
\begin{equation} \label{2023AS-eq:f_g_original_definition}
f(x) : = - c_2 x^{\kappa}  
\quad
\text{and}
\quad
g (x) : = c_3 x^{\rho},
\end{equation} 
satisfying certain assumptions specified subsequently (cf. Assumptions \ref{2023AS-ass:f_g_properties}-\ref{2023AS-ass:modification_bounds}).
At first, we show that the proposed scheme preserves the positivity of the original model unconditionally.
%
%

\begin{lem}[Unconditional positivity-preserving]
\label{lem:positive-preserving}
Let $c_{-1},c_{0},c_{1},c_{2},c_{3}>0$, $\kappa$, $\rho>1$.
Given any initial value $Y_0 = X_0 > 0$, 
the scheme \eqref{2023AS-eq:EUPE_scheme} admits unique, positive numerical solutions 
$\{ Y_n \}_{ n \in \N} $ for any step-size $h = \tfrac{T}{N} > 0$. In other words, the scheme \eqref{2023AS-eq:EUPE_scheme} is positivity preserving unconditionally.
\end{lem}
\textbf{Proof:}
For $x > 0$, we introduce the function
\begin{equation}\label{2023AS-eq:Gx}
    G(x) : = x - c_{-1} h x^{-1}.
\end{equation}
Evidently, the well-posedness of the scheme 
\eqref{2023AS-eq:EUPE_scheme} in the positive domain $(0, \infty)$ is to find a unique positive solution of the implicit equation $G (x) = c$ for any $c \in \R$.
On the one hand, we observe that $c_{-1} h > 0$ and thus
\begin{equation}
\lim_{x \rightarrow 0+} G(x) = - \infty,
\quad
\lim_{x \rightarrow \infty} G(x) = + \infty.
\end{equation}
On the other hand, note that the function $G$ is 
monotonically increasing in the domain $(0, \infty)$ 
since $G'(x) = 1 + c_{-1} h x^{-2} > 0$.
As a direct consequence, for any $c \in \R$, $G (x) = c$ admits a unique root in $(0, +\infty)$ and this validates the desired assertion. 
\qed

Indeed, the unique positive solution determined 
by the proposed scheme \eqref{2023AS-eq:EUPE_scheme} is given by
\begin{equation}
\label{eq:scheme-solution}
Y_{n+1}
=
\frac{ Y_n +  \vartheta_n h +
    g_h(Y_n) \Delta W_n + \sqrt{ (Y_n + \vartheta_n h +
    g_h(Y_n) \Delta W_n )^2 + 4 c_{-1} h } }{2},
\end{equation}
where for short we denote 
\[
\vartheta_n :=  - c_0 + c_1 Y_n + f_h(Y_n).
\]
To prove the bounded moments of the numerical approximations, we formulate the following assumptions.

\begin{assumption}
\label{2023AS-ass:f_g_properties}
    For any $x \in (0, + \infty)$, it holds that
        \begin{equation}\label{2023AS-eq:f_g_bounds}
            \big \vert f_h(x) \big\vert \leq C h^{- \frac12},
            \quad
            \big\vert g_h(x) \big\vert \leq C h^{- \frac12}.
        \end{equation}
\end{assumption}

\begin{assumption}
\label{2023AS-ass:monotone_coupling_condition}
Let $\kappa + 1 \geq 2 \rho$ with $\kappa, \rho > 1$.
    There exists a positive constant $ L (\gamma)$, independent of the step-size $h>0$, such that
        \begin{align}
        \label{2023AS-eq:monotone_coupling_condition}
            \langle x,f_h(x) \rangle + \gamma \vert g_h(x)\vert ^{2} \leq L < \infty,
            \quad
            \forall x \in (0, + \infty),
        \end{align}
where $\gamma \geq \tfrac12 $ is assumed to be arbitrarily large for the case $\kappa + 1 > 2 \rho$ and 
{\color{black}$\tfrac12 \leq \gamma \leq \tfrac{c_2}{c_3^2}$}
for the critical case $\kappa + 1 = 2 \rho$.
\end{assumption}

\begin{assumption}
\label{2023AS-ass:modification_bounds}
There exist some positive {\color{black}constants} $m_1,m_2 \leq 2 \gamma + 1$ such that
        \begin{equation}\label{2023AS-eq:modification_bounds}
            \big\vert f (x) - f_h(x) \big\vert^2 
            \leq C h \vert x \vert^{m_1},
            \quad
            \big\vert g (x) - g_h(x) \big\vert^2
            \leq C h \vert x \vert^{m_2},
            \quad
            \forall x \in (0, +\infty)
        \end{equation}
where $f,g$ are defined by \eqref{2023AS-eq:f_g_original_definition} and $\gamma$ comes from Assumption \ref{2023AS-ass:monotone_coupling_condition}.
\end{assumption}
}

{\color{black}
Next we show an example of modifications which meets all the above assumptions.

\begin{example}
\label{2023AS:eg-tamed_Euler_method}
{\color{black}For $\kappa + 1 \geq 2 \rho$,} 
let us consider an example of a modification of $f,g$ as follows:
\begin{align}
    f_h(x)
    : =
    - \tfrac{c_2 x^{\kappa}}{1+ h^{\alpha}x^{2 \kappa \alpha}},
    \quad
    g_h(x)
    : =
    \tfrac{c_3 x^{\rho}}{1+ h^{\alpha}x^{2 \kappa \alpha}},
\end{align}
where $\alpha$ is a positive constant satisfying $\alpha \geq \half$.
%
Here, the choice of $f_h$ and $g_h$ are nothing but a kind of taming strategy, which was firstly proposed in \cite{hutzenthaler2012strong} and later extensively studied 
by many authors, see, e.g., \cite{wang2013tamed,brehier2020approximation,sabanis2013note,sabanis2016euler,liu2023higher,Lord2024convergence}.
To validate Assumption \ref{2023AS-ass:f_g_properties}, 
{\color{black} noticing that $2 \kappa \alpha \geq \kappa > \rho$,}
simple calculations infer that for any $x \in (0, +\infty)$
\begin{align}
    \vert f_h(x) \vert^{2 \alpha}
    \leq
    \tfrac{ c_2^{2 \alpha} x^{2 \kappa \alpha}}{1+ h^{\alpha}x^{2 \kappa \alpha}}
    \leq 
    c_2^{2 \alpha} h^{-\alpha},
    \quad
    \vert g_h(x) \vert^{\frac{2 \kappa \alpha}{\rho}}
    \leq
    \tfrac{ c_3^{\frac{2 \kappa \alpha}{\rho}} x^{2 \kappa \alpha}}{1 + h^{\alpha}x^{2 \kappa \alpha}}
    \leq
    c_3^{\frac{2 \kappa \alpha}{\rho}} h^{-\alpha}.
\end{align}
Thus, we have
\begin{align}
    \vert f_h(x) \vert 
    \leq 
    c_2 h^{- \frac12}
    \quad
    \text{and}
    \quad
    \vert g_h(x) \vert 
    \leq 
    c_3 h^{- \frac{\rho}{2 \kappa}}
    \leq
    c_3 T^{\frac12 - \frac{\rho}{2 \kappa} }
    h^{-\frac12},
\end{align}
which confirms Assumption \ref{2023AS-ass:f_g_properties}.
To check Assumption \ref{2023AS-ass:monotone_coupling_condition}, we consider two possibilities:  $\kappa + 1 > 2 \rho$ and $\kappa + 1 = 2 \rho$. In the former case, one can easily observe that for any $\gamma > 0$ there exists a constant $L > 0$ such that
\begin{equation}
\begin{aligned}
    \langle x,f(x) \rangle + \gamma \vert g(x)\vert ^{2}
    & =
    -c_2 x^{\kappa + 1} + \gamma c_3^2 x^{2 \rho}
    \leq
    L,
\end{aligned}
\end{equation}
and thus
\begin{equation}
\begin{aligned}
    \langle x,f_h(x) \rangle + \gamma \vert 
    g_h(x)\vert ^{2}
    =
    \tfrac
            { 
                    -c_{2} x^{\kappa + 1}
                +
                \gamma c_{3}^2 x^{2\rho}
            }
            { 1+h^{\alpha}x^{2 \kappa \alpha}}
    \leq
	\tfrac
            {  L }
            { 1+h^{\alpha}x^{2 \kappa \alpha}}
        \leq L.
\end{aligned}
\end{equation}
For the latter case $\kappa + 1 = 2 \rho$, one
needs $\gamma \leq \tfrac{c_2}{c_3^2}$
to ensure the above two inequalities hold,
which confirm Assumption \ref{2023AS-ass:monotone_coupling_condition}.
Finally, it can be deduced that
\begin{align}
    \big \vert f(x) - f_h(x) \big\vert
    & =
    \tfrac{c_2 h^{\alpha} x^{(2\alpha + 1) \kappa}}{1+ h^{\alpha}x^{2 \kappa \alpha}}
    \leq
    c_2 h^{\alpha} x^{(2\alpha + 1) \kappa},\\
    \big\vert g(x) - g_h(x) \big\vert
    & =
    \tfrac{c_3 h^{\alpha} x^{{\color{black}2\kappa\alpha + \rho}}}{1+ h^{\alpha}x^{2 \kappa \alpha}}
    \leq c_3 h^{\alpha} x^{{\color{black}2\kappa\alpha + \rho}}.
\end{align}
Let $m_1 = 2(2\alpha + 1)\kappa$ and $m_2 = 2({\color{black}2\kappa\alpha + \rho})$.
{\color{black}Clearly, $m_2 \leq 4 \alpha \kappa + \kappa + 1 \leq m_1$.}
Since $\alpha \geq \half$, \eqref{2023AS-eq:modification_bounds} holds true 
on an additional condition $ 2 (2\alpha + 1) \kappa \leq  2\gamma + 1$.

\end{example}
}

{\color{black}
In what follows we establish the integrability of the numerical approximations.
\begin{lem}\label{2023AS-lem:Y_integrability_gamma}
Let Assumptions \ref{2023AS-ass:f_g_properties} and \ref{2023AS-ass:monotone_coupling_condition} hold. 
Then for any $2 \leq p \leq 2 \gamma + 1$ 
and for any step-size $h>0$
it holds that
\begin{equation}
    \sup_{0 \leq n \leq N} 
    \E \big[ \vert Y_{n} \vert ^{p} \big]
    \leq
    C ( 1 + \vert x_0 \vert^p )
    < \infty.
\end{equation}
\end{lem}
}
\textbf{Proof:}
To begin with, the proposed scheme \eqref{2023AS-eq:EUPE_scheme} can be rewritten as
\begin{equation}
    Z_{n+1}
    =
    Y_{n}
    +
    \big( - c_0 + c_1 Y_n + f_h(Y_n) \big) h
    +
    g_h(Y_n) \Delta W_n,
    \quad
    n = 0,1,...,N-1,
\end{equation}
where we denote 
\begin{equation}
    Z_{n+1} : = Y_{n+1} - c_{-1} h Y_{n+1}^{-1}.
\end{equation}
Next, we introduce a time-continuous version of $Z_{n+1}$ in each time interval $[t_n, t_{n+1}]$ as follows:
\begin{equation}\label{2023AS-eq:Ztn_continuous_definition}
    \begin{aligned}
        \widetilde Z_{t}^{(n)} 
        : = 
        Y_{n}
        +
        \big( - c_0 + c_1 Y_n + f_h(Y_n) \big) (t - t_n)
        +
        g_h(Y_{n}) (W_{t}-W_{t_{n}}),
        \quad
        t \in [t_n,t_{n+1}].
    \end{aligned}
\end{equation}
Note that $\widetilde Z_{t_n}^{(n)} = Y_n \neq Z_n = \widetilde Z_{t_n}^{(n-1)}$.
{\color{black}
Owing to \eqref{2023AS-eq:f_g_bounds}, 
a rough moment bound can be derived for $\widetilde Z_t^{(n)}$.
Indeed, by the Jensen inequality, 
for any $p \geq 1$, we have
\begin{align}\label{2023AS-eq:Zt_Jensen_rough_bound}
    \sup_{t_n \leq t \leq t_{n+1}}
    \E \bigg[ \Big(
        \varepsilon 
        + 
        \big\vert 
        \widetilde Z_t^{(n)} 
        \big\vert ^2 
    \Big)^p \bigg]
    \leq
    \hat C
    \E \Big[ 
        \big\vert 
        Y_n 
        \big\vert ^{2p}
    \Big]
    +
    \tilde C_{\varepsilon},
\end{align}
where $\varepsilon \geq 1 + 2 c_{-1} h > 1$ is a positive constant specified later and $\hat C$$, \tilde C_{\varepsilon}$ are positive constants independent of $h$.
{\color{black}
By noting $\varepsilon \geq 2 c_{-1} h$, we infer
\begin{align}
    \E \bigg[ \Big(
        \varepsilon 
        + 
        \big\vert 
        \widetilde Z_{t_{n+1}}^{(n)} 
        \big\vert ^2 
    \Big)^p \bigg]
    & =
    \E \bigg[ \Big(
        \varepsilon 
        + 
        \big\vert 
        Y_{n+1} - c_{-1} h Y_{n+1}^{-1} 
        \big\vert ^2 
    \Big)^p \bigg]
    \nonumber
    \\
   & =
    \E \bigg[ \Big(
        \varepsilon 
        + 
        \vert Y_{n+1} \vert^2
        +
        c_{-1}^2 h^2 \vert Y_{n+1} \vert^{-2}
        -
        2 c_{-1} h
    \Big)^p \bigg]
    \nonumber
    \\
   &  \geq
   \E \big[ 
        \big\vert 
        Y_{n+1}
        \big\vert ^{2p}
    \big]
    .
\end{align}
}
Therefore, we deduce from \eqref{2023AS-eq:Zt_Jensen_rough_bound} that
\begin{align}
    \E \big[ 
        \big\vert 
        Y_{n+1}
        \big\vert ^{2p}
    \big]
    \leq
    \hat C
    \E \big[ 
        \big\vert 
        Y_n 
        \big\vert ^{2p}
    \big]
    +
    \tilde C_{\varepsilon}.
\end{align}
By iteration, we get, for any $p \geq 1$ and $0 \leq n \leq N$,
{\color{black}
\begin{align} \label{2023As-eq:Y-rough-bound}
    \E \big[ 
        \big\vert 
        Y_n
        \big\vert ^{2p}
    \big]
    \leq
    \hat C^{n} \E \big[ 
        \big\vert 
        Y_0
        \big\vert ^{2p}
    \big]
    +
    \tilde C_{\varepsilon}
    \sum_{i=0}^{n-1}  \hat C^i
    <
    \infty,
\end{align}
}
which, together with \eqref{2023AS-eq:Zt_Jensen_rough_bound}, infers that,
for any $p \geq 1$,
\begin{align}\label{2023AS-eq:rough_bound}
    \sup_{0 \leq n \leq N-1}
    \sup_{t_n \leq t \leq t_{n+1}}
    \E \big[ \big(
        \varepsilon 
        + 
        \big\vert 
        \widetilde Z_t^{(n)} 
        \big\vert ^2 
    \big)^p \big]
    < \infty.
\end{align}
We emphasize that the above rough moment bound, depending on $N$, 
is not proved to be uniformly bounded with respect to $N$. Nevertheless, it suffices to allow the use of the Gronwall lemma
later to arrive at uniform moment bounds.
%
To do this, we denote 
\begin{equation}\label{2023AS-eq:V_definition}
    V_{\varepsilon,q}(x) 
    : = 
    \big( 
     \varepsilon
    +
    \vert x \vert ^2
    \big)^q,
    \quad
    q,x \in \R,
\end{equation}
where we recall
$\varepsilon \geq 1 + 2 c_{-1}h$.
It can be easily checked that for any $q,q' \in \R$,
\begin{equation}\label{2023AS-eq:V_qq'_property}
    V_{\varepsilon,q}(x) \cdot V_{\varepsilon,q'}(x)
    =
    V_{\varepsilon, q+q'}(x),
    \quad
    \big( V_{\varepsilon,q}(x) \big)^{q'} 
    = 
    V_{\varepsilon, q \cdot q'}(x).
\end{equation}
Since $\varepsilon \geq 1 + 2 c_{-1}h > 1$,
one can directly check that
\begin{equation}\label{2023AS-eq:V_q<q'}
    V_{\varepsilon, q} \leq V_{\varepsilon,q'},
    \quad
    q \leq q'.
\end{equation}
Also, elementary calculations, for $q \neq 0$, show
\begin{equation}\label{2023AS-eq:V_derivates}
\begin{aligned}
    V_{\varepsilon,q}'(x) 
    & = 
    2 q x 
    \big( 
        \varepsilon 
        +
        \vert x \vert ^2
    \big)^{q-1}
    =
    2 q x V_{\varepsilon,q-1}(x),\\
    V_{\varepsilon,q}''(x)
    & =
    2 q \big( 
        \varepsilon 
        +
        \vert x \vert ^2
    \big)^{q-2}
    \big[
        \varepsilon 
        +
        (2q - 1) \vert x \vert ^2
    \big]
    & \leq
    2q \cdot 
    \max{\{ 1, 2q-1 \}} 
    \cdot 
    V_{\varepsilon, q-1}(x)
\end{aligned}
\end{equation}
and that
\begin{equation}\label{2023AS-eq:xV_leq}
    x
    \cdot
    V_{\varepsilon,q}(x)
    \leq
    \big( \varepsilon + \vert x \vert^2 \big)^{\frac12}
    \cdot
    V_{\varepsilon,q}(x)
    =
    V_{\varepsilon,\frac12}(x)
    \cdot
    V_{\varepsilon,q}(x)
    =
    V_{\varepsilon,q+\frac12}(x).
\end{equation}
Equipped with these properties, we can start the estimation of $\widetilde Z_{t}^{(n)}$.
The It\^{o} formula together with \eqref{2023AS-eq:V_derivates} and \eqref{2023AS-eq:xV_leq} yields 
{\color{black}that for any $q \geq 1$,}
\begin{equation}\label{2023AS-eq:Vq_Ito}
\begin{aligned}
    V_{\varepsilon,q}( \widetilde Z_t^{(n)})
    & = 
    V_{\varepsilon,q}( \widetilde Z_{t_n}^{(n)})
    +
    \int_{t_n}^t
    \bigg[
    V_{\varepsilon,q}'( \widetilde Z_s^{(n)} )
    \big( - c_0 + c_1 Y_n + f_h(Y_n) \big)
    +
    \half V_{\varepsilon,q}''( \widetilde Z_s^{(n)} )
    \vert g_h(Y_n) \vert^2
    \bigg]
    \, \dd s\\
    & \quad 
    +
    \int_{t_n}^t
    V_{\varepsilon,q}'( \widetilde Z_s^{(n)}) g_h(Y_n)
    \, \dd W_s\\
    & \leq
    V_{\varepsilon,q}( \widetilde Z_{t_n}^{(n)})
    +
    \int_{t_n}^t
    \bigg[
    2 q \widetilde Z_s^{(n)}  
    V_{\varepsilon,q-1}(\widetilde Z_{s}^{(n)})
    \big( -c_0 + c_1 Y_n + f_h(Y_n) \big)\\
    & \quad 
    +
    q (2q - 1) 
    V_{\varepsilon,q-1}(\widetilde Z_{s}^{(n)})
    \vert g_h(Y_n) \vert^2
    \bigg]
    \, \dd s
    +
    \int_{t_n}^t
    2 q \widetilde Z_s^{(n)}  
    V_{\varepsilon,q-1}(\widetilde Z_{s}^{(n)})
    g_h(Y_n)
    \, \dd W_s\\
    & \leq
    V_{\varepsilon,q}(\widetilde Z_{t_n}^{(n)})
    +
    \int_{t_n}^t
    2 q 
    V_{\varepsilon,q-\frac12}(\widetilde Z_{s}^{(n)})
    \cdot
    \big\vert -c_0 + c_1 Y_n \big\vert
    \, \dd s\\
    & \quad 
    +
    \int_{t_n}^t
    2 q 
    V_{\varepsilon,q-1}(\widetilde Z_{s}^{(n)})
    \Big[ \big\langle Y_n, f_h(Y_n) \big\rangle + 
    \tfrac{2q - 1}{2} \vert g_h(Y_n) \vert^2 \Big]
    \, \dd s\\
    & \quad 
    +
    \int_{t_n}^t
    2 q 
    V_{\varepsilon,q-1}(\widetilde Z_{s}^{(n)})
    \big( \widetilde Z_{s}^{(n)} - Y_n \big) f_h(Y_n)
    \, \dd s\\
    & \quad 
    +
    \int_{t_n}^t
        2 q \widetilde Z_s^{(n)}  
        V_{\varepsilon,q-1}(\widetilde Z_{s}^{(n)})
        g_h(Y_n)
    \, \dd W_s\\
    & =
    V_{\varepsilon,q}(\widetilde Z_{t_n}^{(n)})
    +
    \int_{t_n}^t
    2 q 
    V_{\varepsilon,q-\frac12}(\widetilde Z_{s}^{(n)})
    \cdot
    \big\vert -c_0 + c_1 Y_n \big\vert
    \, \dd s\\
    & \quad 
    +
    \int_{t_n}^t
    2 q 
    V_{\varepsilon,q-1}(\widetilde Z_{s}^{(n)})
    \Big[ \big\langle Y_n, f_h(Y_n) \big\rangle + 
    \tfrac{2q - 1}{2} \vert g_h(Y_n) \vert^2 \Big]
    \, \dd s\\
    & \quad 
    +
    \int_{t_n}^t
    2 q 
    V_{\varepsilon,q-1}(\widetilde Z_{s}^{(n)})
    \int_{t_n}^s 
        f_h(Y_n)
        \big( - c_0 + c_1 Y_n + f_h(Y_n) \big) 
    \, \dd r
    \, \dd s\\
    & \quad 
    +
    \int_{t_n}^t
    2 q
    V_{\varepsilon,q-1}(\widetilde Z_{s}^{(n)})
    \int_{t_n}^s 
        f_h(Y_n)
        g_h(Y_n)
    \, \dd W_r
    \, \dd s\\
    & \quad 
    +
    \int_{t_n}^t
        2 q \widetilde Z_s^{(n)}  
        V_{\varepsilon,q-1}(\widetilde Z_{s}^{(n)})
        g_h(Y_n)
    \, \dd W_s,
\end{aligned}
\end{equation}
where we also used \eqref{2023AS-eq:Ztn_continuous_definition} in the last step.
}
{\color{black}
In view of \eqref{2023AS-eq:monotone_coupling_condition}, for any $q \leq \gamma + \half$, 
{\color{black}where $\gamma$ is assumed to be arbitrarily large for the case $\kappa + 1 > 2 \rho$ and $\tfrac12 \leq \gamma \leq \tfrac{c_2}{c_3^2}$ for the critical case $\kappa + 1 = 2 \rho$,}
one deduces that
\begin{align}\label{2023AS-eq:coupled_monotone_2q-1}
    \big\langle Y_n, \, f_h(Y_n) \big\rangle
        +
        \tfrac{2q-1}{2} \vert g_h(Y_n) \vert^2
    \leq L,
\end{align}
which, together with \eqref{2023AS-eq:f_g_bounds}, further implies that
\begin{equation}
\begin{aligned}
    f_h(Y_n) \big( - c_0 + c_1 Y_n + f_h(Y_n) \big) 
    & \leq
    -c_0 f_h(Y_n) + c_1 L + \vert f_h(Y_n) \vert^2
    \\
    & \leq
    C (1 + \vert f_h(Y_n) \vert^2 )
    \\ & 
    \leq
    C(1 + h^{-1}).
\end{aligned}
\end{equation}
Taking these estimates into account, one can derive from \eqref{2023AS-eq:V_q<q'} and \eqref{2023AS-eq:Vq_Ito} that 
\begin{equation}
\begin{aligned}
    V_{\varepsilon,q}( \widetilde Z_t^{(n)})
    & \leq
    V_{\varepsilon,q}(\widetilde Z_{t_n}^{(n)})
    +
    \int_{t_n}^t
    2 q c_0
    V_{\varepsilon,q}(\widetilde Z_{s}^{(n)})
    \, \dd s
    +
    \int_{t_n}^t
    2 q c_1
    V_{\varepsilon,q - \half}(\widetilde Z_{s}^{(n)})
    \cdot
    Y_n
    \, \dd s\\
    & \quad 
    +
    \int_{t_n}^t
    2 q 
    V_{\varepsilon,q-1}(\widetilde Z_{s}^{(n)})
    ( L+C )
    \, \dd s\\
    & \quad 
    +
    \int_{t_n}^t
    2 q
    V_{\varepsilon,q-1}(\widetilde Z_{s}^{(n)})
    \int_{t_n}^s 
        f_h(Y_n)
        g_h(Y_n)
    \, \dd W_r
    \, \dd s\\
    & \quad 
    +
    \int_{t_n}^t
        2 q \widetilde Z_s^{(n)}  
        V_{\varepsilon,q-1}(\widetilde Z_{s}^{(n)})
        g_h(Y_n)
    \, \dd W_s.
\end{aligned}
\end{equation}
Applying the Young inequality twice leads to
\begin{equation}\label{2023AS-eq:Vq_after_Young}
\begin{aligned}
    V_{\varepsilon,q}(\widetilde Z_{t}^{(n)})
    & \leq
    V_{\varepsilon,q}(\widetilde Z_{t_n}^{(n)})
    +
    \int_{t_n}^t
    2 q c_0
    V_{\varepsilon,q}(\widetilde Z_{s}^{(n)})
    \, \dd s
    +
    \int_{t_n}^t
        c_1 (2q-1) 
        V_{\varepsilon,q}(\widetilde Z_{s}^{(n)})
    \, \dd s
    +
    \int_{t_n}^t
        c_1 \vert Y_n \vert ^{2q}
    \, \dd s\\
    & \quad 
    +
    \int_{t_n}^t
    (2q-2) 
    V_{\varepsilon,q}(\widetilde Z_{s}^{(n)})
    \, \dd s
    +
    \int_{t_n}^t
        2 (L + C)^q
    \, \dd s\\
    & \quad 
    +
    \int_{t_n}^t
        2q 
        V_{\varepsilon,q-1}(\widetilde Z_{s}^{(n)})
        \int_{t_n}^s f_h(Y_n) g_h(Y_n) \, \dd W_r
        \, \dd s\\
    & \quad 
    +
    \int_{t_n}^t
    2q 
    \widetilde Z_s^{(n)}
    V_{\varepsilon,q-1}(\widetilde Z_{s}^{(n)})
    g_h(Y_n)
    \, \dd W_s.
\end{aligned}
\end{equation}
{\color{black}Recalling \eqref{2023AS-eq:Ztn_continuous_definition} and \eqref{2023AS-eq:V_definition}, one knows that the process 
$\widetilde Z_s^{(n)}$ is $\mathcal{F}_{s}$-measurable,
$Y_n$
is $\mathcal{F}_{t_n}$-measurable by construction
and both are integrable in view of \eqref{2023AS-ass:f_g_properties} and \eqref{2023AS-eq:rough_bound}. This ensures that the following expectation of the It\^{o} stochastic integral vanishes:}
\begin{equation}\label{2023AS-eq:expectation_vanish}
    \E \Bigg[
        \int_{t_n}^t
        2q 
        \widetilde Z_s^{(n)}
        V_{\varepsilon,q-1}(\widetilde Z_{s}^{(n)})
        g_h(Y_n)
        \, \dd W_s
    \Bigg]
    =0.
\end{equation}
{\color{black}
By observing that
\begin{equation}
    \vert Y_n \vert ^{2q}
    \leq
    \big( \varepsilon + \vert Y_n \vert^2 \big)^q
    =
    V_{\varepsilon,q}(Y_n)
    =
    V_{\varepsilon,q}(\widetilde Z_{t_n}^{(n)}),
\end{equation}
}
we take expectations on both sides of \eqref{2023AS-eq:Vq_after_Young} to show
\begin{equation}\label{2023AS-eq:Zt_before_Ito_on_Z_s^2p-2}
\begin{aligned}
    \E \Big[ 
        V_{\varepsilon,q}(\widetilde Z_{t}^{(n)}) 
    \Big]
    & \leq
    (1 + c_1 h)\E \Big[ 
        V_{\varepsilon,q}(\widetilde Z_{t_n}^{(n)})
    \Big]
    +
    (2 q (c_0+c_1+1) - c_1 - 2) 
    \E \bigg[ 
        \int_{t_n}^t
        V_{\varepsilon,q}(\widetilde Z_{s}^{(n)})
        \, \dd s
    \bigg]
    +
    C h\\
    & \quad 
    +
    \E \bigg[
        \int_{t_n}^t
            2q 
            V_{\varepsilon,q-1}(\widetilde Z_{s}^{(n)})
            \int_{t_n}^s f_h(Y_n) g_h(Y_n) \, \dd W_r
        \, \dd s
    \bigg].
\end{aligned}
\end{equation}
}
The case $q = 1$ directly results in 
\begin{equation}
\begin{aligned}
    \E \Big[ 
        V_{\varepsilon,1}(\widetilde Z_{t}^{(n)}) 
    \Big]
    & \leq
    (1 + c_1 h)\E \Big[ 
        V_{\varepsilon,1}(\widetilde Z_{t_n}^{(n)})
    \Big]
    +
    {\color{black}(2c_0 + c_1)} \cdot
    \E \bigg[ 
        \int_{t_n}^t
        V_{\varepsilon,1}(\widetilde Z_{s}^{(n)})
        \, \dd s
    \bigg]
    +
    C h.
\end{aligned}
\end{equation}
Now, we focus on the case $q > 1$.
Applying the It\^{o} formula again on 
$
V_{\varepsilon,q-1}(\widetilde Z_s^{(n)}),
$ 
we get
{\color{black}
\begin{equation}\label{2023AS-eq:Ito_on_Zs^2p-2}
\begin{aligned}
    V_{\varepsilon,q-1}(\widetilde Z_s^{(n)})
    & =
    V_{\varepsilon,q-1}(\widetilde Z_{t_n}^{(n)})
    +
    \int_{t_n}^s
        V_{\varepsilon, q-1}'(\widetilde Z_r^{(n)})
        \cdot
        \big( - c_0 + c_1 Y_n \big)
    \, \dd r\\
    & \quad 
    +
    \int_{t_n}^s
    \bigg[
        V_{\varepsilon, q-1}'(\widetilde Z_r^{(n)})
        f_h(Y_n)
        + 
        \half V_{\varepsilon, q-1}''(\widetilde Z_r^{(n)})
        \vert g_h(Y_n) \vert^2
    \bigg]
    \, \dd r\\
    & \quad +
    \int_{t_n}^s
        V_{\varepsilon, q-1}'(\widetilde Z_r^{(n)})
        g_h(Y_n) 
    \, \dd W_r.
\end{aligned}
\end{equation}
{\color{black}Similar to \eqref{2023AS-eq:expectation_vanish}, we note that
\begin{equation}
\begin{aligned}
    \E \bigg[
        \int_{t_n}^t
            2q 
            V_{\varepsilon,q-1}(\widetilde Z_{t_n}^{(n)})
            \int_{t_n}^s f_h(Y_n) g_h(Y_n) \, \dd W_r
        \, \dd s
    \bigg]
    & =
    2q 
        \int_{t_n}^t 
            \E \bigg[
            \int_{t_n}^s V_{\varepsilon,q-1}(\widetilde Z_{t_n}^{(n)})
            f_h(Y_n) g_h(Y_n) \, \dd W_r
                \bigg]
        \, \dd s
    \\
    &
    =
    0.
\end{aligned}
\end{equation}
}
Taking this into account and plugging \eqref{2023AS-eq:Ito_on_Zs^2p-2} into \eqref{2023AS-eq:Zt_before_Ito_on_Z_s^2p-2},
we have
\begin{equation}
\begin{aligned}
    & \E \Big[ 
        V_{\varepsilon,q}(\widetilde Z_{t}^{(n)})
    \Big]\\
    & \leq
    (1 + c_1 h)\E \Big[ 
        V_{\varepsilon,q}(\widetilde Z_{t_n}^{(n)})
    \Big]
    +
    (2 q (c_0+c_1+1) - c_1 - 2) \cdot
    \E \bigg[ 
        \int_{t_n}^t
        V_{\varepsilon,q}(\widetilde Z_{s}^{(n)})
        \, \dd s
    \bigg]
    +
    C h\\
    & \quad 
    +
    \underbrace{
    \E \bigg[
        \int_{t_n}^t
            2q 
            \int_{t_n}^s
                V_{\varepsilon, q-1}'(\widetilde Z_r^{(n)})
                \cdot
                \big( - c_0 + c_1 Y_n \big)
            \, \dd r
            \int_{t_n}^s f_h(Y_n) g_h(Y_n) \, \dd W_r
        \, \dd s
    \bigg]
    }_{=: D_1}\\
    & \quad 
    +
    \underbrace{
    \E \bigg[
        \int_{t_n}^t
            2q 
            \int_{t_n}^s
            \bigg[
                V_{\varepsilon, q-1}'(\widetilde Z_r^{(n)})
                f_h(Y_n)
                + 
                \half V_{\varepsilon, q-1}''(\widetilde Z_r^{(n)})
                \vert g_h(Y_n) \vert^2
            \bigg]
            \, \dd r
            \int_{t_n}^s f_h(Y_n) g_h(Y_n) \, \dd W_r
        \, \dd s
    \bigg]
    }_{=: D_2}\\
    & \quad 
    +
    \underbrace{
    \E \bigg[
        \int_{t_n}^t
            2q 
            \int_{t_n}^s
                V_{\varepsilon, q-1}'(\widetilde Z_r^{(n)})
                g_h(Y_n) 
            \, \dd W_r
            \int_{t_n}^s f_h(Y_n) g_h(Y_n) \, \dd W_r
        \, \dd s
    \bigg].
    }_{=: D_3}
\end{aligned}
\end{equation}
}
{\color{black}
Utilizing the Young inequality of the form
\begin{equation*}
    a \cdot b \leq \tfrac{1}{2q} a^{2q} + \tfrac{2q-1}{2q} b^{\frac{2q}{2q-1}}
    =
    \tfrac{1}{2q}\Big(a^{2q} + (2q-1) b^{\frac{2q}{2q-1}} \Big),
\end{equation*}
we first derive
\begin{equation}\label{2023AS-eq:D1-early}
\begin{aligned}
    D_1
    & =
    4q(q-1) \E \bigg[ 
        \int_{t_n}^t\
            \int_{t_n}^s
                \widetilde Z_r^{(n)}
                V_{\varepsilon,q-2}(\widetilde Z_r^{(n)})
                \cdot
                \big( - c_0 + c_1 Y_n \big)
            \, \dd r
            \int_{t_n}^s f_h(Y_n) g_h(Y_n) \, \dd W_r
        \, \dd s
    \bigg]
    \\
    & \leq
    2(q-1) \E \bigg[ 
        \int_{t_n}^t
            \big\vert - c_0 + c_1 Y_n \big\vert^{2q}
        \, \dd s
    \bigg]\\
    & \quad
    +
    2(q-1)(2q-1) \E \bigg[ 
        \int_{t_n}^t
            \Big\vert
                \int_{t_n}^s
                    \widetilde Z_r^{(n)}
                    V_{\varepsilon,q-2}(\widetilde Z_r^{(n)})
                \, \dd r
                \int_{t_n}^s
                    f_h(Y_n) g_h(Y_n)
                \, \dd W_r
            \Big\vert^{\frac{2q}{2q - 1}}
        \, \dd s
    \bigg].
\end{aligned}
\end{equation}
Further, applying the Young inequality of another form
\begin{align*}
    a \cdot b \leq \tfrac{2q-2}{2q-1} a^{\frac{2q-1}{2q-2}}
    +
    \tfrac{1}{2q-1} b^{2q-1}
    =
    \tfrac{1}{2q-1} \big( (2q-2)a^{\frac{2q-1}{2q-2}}
    + b^{2q-1} \big),
\end{align*}
we get
\begin{equation}
\begin{aligned}
    D_1 
    & \leq
    2(q-1) \E \bigg[ 
        \int_{t_n}^t
            \big\vert - c_0 + c_1 Y_n \big\vert^{2q}
        \, \dd s
    \bigg]\\
    & \quad
    +
    2(q-1)(2q-2) \E \bigg[ 
        \int_{t_n}^t
            \Big\vert
                h^{-1}
                \int_{t_n}^s
                    \widetilde Z_r^{(n)}
                    V_{\varepsilon,q-2}(\widetilde Z_r^{(n)})
                \, \dd r
            \Big\vert^{\frac{2q}{2q - 2}}
        \, \dd s
    \bigg]\\
    & \quad
    +
    2(q-1) \E \bigg[ 
        \int_{t_n}^t
            \Big\vert
                h
                \int_{t_n}^s
                    f_h(Y_n) g_h(Y_n)
                \, \dd W_r
            \Big\vert^{2q}
        \, \dd s
    \bigg].
\end{aligned}
\end{equation}
For the first term, we can directly deduce that
\begin{equation*}
    \vert - c_0 + c_1 Y_n \vert ^{2q}
    \leq
    C
    \big( \varepsilon + \vert Y_n \vert^2 \big)^q
    =
    C V_{\varepsilon,q}(Y_n)
    =
    C V_{\varepsilon,q}(\widetilde Z_{t_n}^{(n)}).
\end{equation*}
For the second term, utilizing \eqref{2023AS-eq:V_q<q'}, \eqref{2023AS-eq:xV_leq} and the H\"{o}lder inequality, we get
\begin{align*}
    & \E \bigg[ 
        \int_{t_n}^t
            \Big\vert
                h^{-1}
                \int_{t_n}^s
                    \widetilde Z_r^{(n)}
                    V_{\varepsilon,q-2}(\widetilde Z_r^{(n)})
                \, \dd r
            \Big\vert^{\frac{2q}{2q - 2}}
        \, \dd s
    \bigg]\\
    & \leq
    \int_{t_n}^t
    \E \bigg[ 
            \Big\vert
                h^{-1}
                \int_{t_n}^s
                    V_{\varepsilon,q-1}(\widetilde Z_r^{(n)})
                \, \dd r
            \Big\vert^{\frac{2q}{2q - 2}}
    \bigg]
    \, \dd s\\
    & \leq
    \int_{t_n}^t
    h^{\frac{q}{q-1} - 1}
    \int_{t_n}^s
    \E\Big[ \big\vert 
        h^{-1}
                    V_{\varepsilon,q-1}(\widetilde Z_r^{(n)})
    \big\vert^{\frac{q}{q-1}} \Big]
    \dd r
    \dd s\\
    & \leq
    \int_{t_n}^t
    \sup_{t_n \leq r \leq s}
    \E\Big[ 
        V_{\varepsilon,q}(\widetilde Z_r^{(n)})
    \Big]
    \dd s.
\end{align*}
For the third term, the moment inequality \cite[Theorem 7.1]{mao2007stochastic} together with \eqref{2023AS-eq:f_g_bounds} directly infers that
\begin{align*}
    \E \bigg[ 
        \int_{t_n}^t
            \Big\vert
                h
                \int_{t_n}^s
                    f_h(Y_n) g_h(Y_n)
                \, \dd W_r
            \Big\vert^{2q}
        \, \dd s
    \bigg]
    \leq
    C h^{q+1}.
\end{align*}
Thus, in summary, we have
\begin{equation}
\begin{aligned}
    D_1 
    \leq
    C h
    \E \Big[ 
        V_{\varepsilon,q}(\widetilde Z_{t_n}^{(n)})
    \Big]
    +
    (2q-2)^2 
    \int_{t_n}^t
        \sup_{t_n \leq r \leq s}
        \E \Big[ 
            V_{\varepsilon,q}(\widetilde Z_r^{(n)}) 
        \Big]
    \dd s
    + 
    C h^{q+1}.
\end{aligned}
\end{equation}
}
Since the estimation of $D_3$ is similar and easy, 
we present it before treating $D_2$.
By the generalized It\^{o} isometry \cite[Chapter 4, B]{evans2006introduction}, the Young inequality as well as \eqref{2023AS-eq:f_g_bounds}, \eqref{2023AS-eq:V_q<q'}, \eqref{2023AS-eq:xV_leq}, the H\"{o}lder inequality, and the moment inequality {\color{black}\cite[Theorem 7.1]{mao2007stochastic}}, we deduce that
\begin{equation}\label{2023AS-eq:D4}
\begin{aligned}
    D_3
    & =
    4q(q-1)
    \E \bigg[
        \int_{t_n}^t
            \int_{t_n}^s
                \widetilde Z_r^{(n)}
                V_{\varepsilon,q-2}(\widetilde Z_{r}^{(n)})
                f_h(Y_n) \vert g_h(Y_n) \vert^2 
            \, \dd r
        \, \dd s
    \bigg]\\
    & \leq
    4(q-1)
    \E \bigg[
        \int_{t_n}^t
        \bigg(
        (q-1)
        \Big\vert
            h^{-1}
            \int_{t_n}^s
                V_{\varepsilon,q-1}(\widetilde Z_{r}^{(n)})
            \, \dd r
        \Big\vert^{\frac{q}{q-1}}    
        +
        \Big\vert
            h
            \int_{t_n}^s f_h(Y_n) \vert g_h(Y_n) \vert^2 \, \dd r
        \Big\vert^{q}
        \bigg)
        \, \dd s
    \bigg]\\
    & \leq
    (2q-2)^2
    \int_{t_n}^t
        \sup_{t_n \leq r \leq s}
        \E \Big[ 
            V_{\varepsilon,q}(\widetilde Z_{r}^{(n)}) 
        \Big]
    \dd s
    + 
    C h^{\frac{q}{2} + 1}.
\end{aligned}
\end{equation}
As for $D_2$, noting
$    \max {\{1, 2q - 3\}}
    <
    2q - 1$
and using \eqref{2023AS-eq:V_derivates},
one can easily show that for $q > 1$,
\begin{equation}
\begin{aligned}
    D_2
    & \leq
    2q 
    \E \Bigg[
        \int_{t_n}^t
            \bigg\vert
            \int_{t_n}^s
            \Big(
                V_{\varepsilon, q-1}'(\widetilde Z_r^{(n)})
                f_h(Y_n)
                + 
                \half V_{\varepsilon, q-1}''(\widetilde Z_r^{(n)})
                \vert g_h(Y_n) \vert^2
            \Big)
            \, \dd r
            \bigg\vert
            \cdot
            \bigg\vert
            \int_{t_n}^s f_h(Y_n) g_h(Y_n) \, \dd W_r
            \bigg\vert
        \, \dd s
    \Bigg]\\
    & \leq
    2q 
    \E \Bigg[
        \int_{t_n}^t
            \bigg\vert
            \int_{t_n}^s
            \Big(
                2(q-1) 
                \widetilde Z_r^{(n)}
                V_{\varepsilon,q-2}(\widetilde Z_r^{(n)})
                f_h(Y_n)\\
            & \qquad
                + 
                (q-1) \max {\{1, 2q - 3 \}}
                V_{\varepsilon,q-2}(\widetilde Z_r^{(n)})
                \vert g_h(Y_n) \vert^2
            \Big)
            \, \dd r
            \bigg\vert
            \cdot
            \bigg\vert
            \int_{t_n}^s f_h(Y_n) g_h(Y_n) \, \dd W_r
            \bigg\vert
        \, \dd s
    \Bigg]\\
    & <
    4q (q-1)
    \E \Bigg[
        \int_{t_n}^t
            \bigg\vert
            \int_{t_n}^s
                V_{\varepsilon,q-2}(\widetilde Z_r^{(n)})
                \Big[
                    \big\langle Y_n,f_h(Y_n) \big\rangle
                    + 
                    \tfrac{2q-1}{2}
                    \vert g_h(Y_n) \vert^2
                \Big]
            \, \dd r
            \bigg\vert\\
            & \quad
            \times
            \bigg\vert
            \int_{t_n}^s f_h(Y_n) g_h(Y_n) \, \dd W_r
            \bigg\vert
        \, \dd s
    \Bigg]\\
    & \quad +
    4q (q-1)
    \E \Bigg[
        \int_{t_n}^t
            \bigg\vert
            \int_{t_n}^s
                V_{\varepsilon,q-2}(\widetilde Z_r^{(n)})
                \big( \widetilde Z_r^{(n)} - Y_n \big)
                f_h(Y_n)
            \, \dd r
            \bigg\vert
            \cdot
            \bigg\vert
            \int_{t_n}^s f_h(Y_n) g_h(Y_n) \, \dd W_r
            \bigg\vert
        \, \dd s
    \Bigg]\\
    & = :
    D_{21} + D_{22}.
\end{aligned}
\end{equation}
Employing the Young inequality, \eqref{2023AS-eq:V_q<q'}, \eqref{2023AS-eq:coupled_monotone_2q-1} on $D_{21}$ as well as the H\"{o}lder inequality and the moment inequality {\color{black}\cite[Theorem 7.1]{mao2007stochastic}}, we get
\begin{equation}\label{2023AS-eq:D2}
\begin{aligned}
    D_{21}
    & \leq
    4L q (q-1) 
    \E \Bigg[
        \int_{t_n}^t
        \bigg(
        \Big\vert
            h^{-1}
            \int_{t_n}^s
                V_{\varepsilon,q-1}( \widetilde Z_r^{(n)})
            \, \dd r
        \Big\vert
        \cdot
        \Big\vert
            h
            \int_{t_n}^s f_h(Y_n) g_h(Y_n) \, \dd W_r
        \Big\vert
        \bigg)
        \, \dd s
    \Bigg]\\
    & \leq
    4 L (q-1)
    \E \Bigg[
        \int_{t_n}^t
        \bigg(
        (q-1)
        \Big\vert
            h^{-1}
            \int_{t_n}^s
                V_{\varepsilon,q-1}( \widetilde Z_r^{(n)})
            \, \dd r
        \Big\vert^{\frac{q}{q - 1}}
        +
        \Big\vert
            h
            \int_{t_n}^s f_h(Y_n) g_h(Y_n) \, \dd W_r
        \Big\vert^{q}
        \bigg)
        \, \dd s
    \Bigg]\\
    & \leq
    L (2q-2)^2
    \int_{t_n}^t
        \sup_{t_n \leq r \leq s}
        \E \Big[ 
            V_{\varepsilon,q}(\widetilde Z_{r}^{(n)}) 
        \Big]
    \dd s
    + 
    C h^{\frac{q}{2}+1}.
\end{aligned}
\end{equation}
With regard to $D_{22}$,
using the Young inequality twice again as well as the H\"{o}lder inequality, the moment inequality, \eqref{2023AS-eq:f_g_bounds}, \eqref{2023AS-eq:Ztn_continuous_definition}, and \eqref{2023AS-eq:V_q<q'} on $D_{22}$, we deduce that
\begin{equation}\label{2023AS-eq:D_22-early}
\begin{aligned}
    D_{22}
    & \leq
    2(q-1)
    \E \Bigg[
        \int_{t_n}^t
        \bigg(
        (2q-1)
        \Big\vert
            h^{-1}
            \int_{t_n}^s
                V_{\varepsilon,q-2}(\widetilde Z_r^{(n)})
                ( \widetilde Z_r^{(n)} - Y_n ) 
            \, \dd r
        \Big\vert^{\frac{2q}{2q - 1}}\\
    & \quad
        +
        \Big\vert
            h
            \int_{t_n}^s f_h(Y_n)^2 g_h(Y_n) \, \dd W_r
        \Big\vert^{2q}
        \bigg)
        \, \dd s
    \Bigg]\\
    & \leq
    \tfrac{2q-2}{h}
    \int_{t_n}^t
    \int_{t_n}^s
    (2q-1)
    \E \bigg[
        \big\vert V_{\varepsilon,q-2}(\widetilde Z_r^{(n)}) \big\vert^{\frac{2q}{2q-1}}
        \cdot
        \big\vert \widetilde Z_{r}^{(n)} - Y_n \big\vert^{\frac{2q}{2q-1}} 
    \bigg]
    \, \dd r \dd s
    +
    C h\\
    & \leq
    \tfrac{(2q-2){\color{black}(2q-1)}}{h}
    \int_{t_n}^t
    \int_{t_n}^s
    \E \bigg[
        \big\vert V_{\varepsilon,q}(\widetilde Z_r^{(n)}) \big\vert^{\frac{2q-2}{2q-1}}
        \cdot
        \big\vert \widetilde Z_{r}^{(n)} - Y_n \big\vert^{\frac{2q}{2q-1}} 
    \bigg]
    \, \dd r \dd s
    +
    C h\\
    & \leq
    \tfrac{(2q-2)^2}{h}
    \int_{t_n}^t
    \int_{t_n}^s
    \E \Big[
        V_{\varepsilon,q}(\widetilde Z_r^{(n)})
    \Big]
    \, \dd r \dd s
    +
    C h\\
    & \quad 
    +
    \tfrac{{\color{black}2q-2}}{h}
    \int_{t_n}^t
    \int_{t_n}^s
    \E \Bigg[\bigg\vert
        \int_{t_n}^{r}
            \big( - c_0 + c_1 Y_n + f_h(Y_n) \big)
        \, \dd r_1
        +
        \int_{t_n}^{r}
            g_h(Y_n)
        \, \dd W_{r_1}
    \bigg\vert^{2q}
    \Bigg]
    \, \dd r \dd s.
\end{aligned}
\end{equation}
{\color{black}
By the Jensen inequality, the H\"{o}lder inequality, the moment inequality \cite[Theorem 7.1]{mao2007stochastic}, and the fact that
\begin{equation*}
    \big\vert
        - c_0 + c_1 Y_n
    \big\vert^{2q}
    \leq
    C \cdot \big( \varepsilon + \vert Y_n \vert^2 \big)^q
    =
    C \cdot \big( \varepsilon + \vert \widetilde Z_{t_n}^{(n)} \vert^2 \big)^q
    =
    C \cdot V_{\varepsilon,q}(\widetilde Z_{t_n}^{(n)}),
\end{equation*}
one can bound the integrand in last term of \eqref{2023AS-eq:D_22-early} as follows:
\begin{align*}
    \E & \Bigg[\bigg\vert
        \int_{t_n}^{r}
            \big( - c_0 + c_1 Y_n + f_h(Y_n) \big)
        \, \dd r_1
        +
        \int_{t_n}^{r}
            g_h(Y_n)
        \, \dd W_{r_1}
    \bigg\vert^{2q}
    \Bigg]
    \\
    & \leq
    2^{2q-1} \E \Bigg[\bigg\vert
        \int_{t_n}^{r}
            \big( - c_0 + c_1 Y_n + f_h(Y_n) \big)
        \, \dd r_1
    \bigg\vert^{2q}
    \Bigg]
    +
    2^{2q-1} \E \Bigg[\bigg\vert
        \int_{t_n}^{r}
            g_h(Y_n)
        \, \dd W_{r_1}
    \bigg\vert^{2q}
    \Bigg]
    \\
    & \leq
    2^{2q-1} (r-t_n)^{2q-1}
    \int_{t_n}^{r}
    \E \Big[\big\vert
        - c_0 + c_1 Y_n + f_h(Y_n)
    \big\vert^{2q}
    \Big]
    \, \dd r_1\\
    & \quad +
    2^{2q-1} (r-t_n)^{q-1}
    \int_{t_n}^{r}
    \E \Big[\big\vert
        g_h(Y_n)
    \big\vert^{2q}
    \Big]
    \, \dd r_1
    \\
    & \leq
    2^{4q-2} (r-t_n)^{2q}
    \E \Big[\big\vert
        - c_0 + c_1 Y_n
    \big\vert^{2q}
    \Big]+
    2^{4q-2} (r-t_n)^{2q}
    \E \Big[\big\vert
        f_h(Y_n)
    \big\vert^{2q}
    \Big]\\
    & \quad 
    +
    2^{2q-1} (r-t_n)^{q}
    \E \Big[\big\vert
        g_h(Y_n)
    \big\vert^{2q}
    \Big]
    \\
    & \leq
    2^{4q-2} (r-t_n)^{2q}
    C \cdot
    \E \Big[ V_{\varepsilon,q}(\widetilde Z_{t_n}^{(n)})
    \Big]
    +
    2^{4q-2} (r-t_n)^{2q}
    h^{-q}
    +
    2^{2q-1} (r-t_n)^{q}
    h^{-q}.
\end{align*}
Plugging this into \eqref{2023AS-eq:D_22-early} shows
}
\begin{equation}
\begin{aligned}
\label{2023AS-eq:D_22}
{\color{black}
D_{22}
 \leq
    (2q-2)^2 \int_{t_n}^t
        \sup_{t_n \leq r \leq s}
        \E \Big[ 
            V_{\varepsilon,q}(\widetilde Z_r^{(n)})
        \Big]
    \dd s
    +
    C h^{2q+1} \E \Big[ V_{\varepsilon,q}(\widetilde Z_{t_n}^{(n)}) \Big]
    + 
    C h}.
\end{aligned}
\end{equation}
Combining all the above estimations and taking supremum reveal that
\begin{equation}
\begin{aligned}
    \sup_{t_n \leq r \leq t}
    \E \Big[ 
        V_{\varepsilon,q}(\widetilde Z_r^{(n)}) 
    \Big]
    & \leq
    (1 + C h)\E \Big[ V_{\varepsilon,q}(\widetilde Z_{t_n}^{(n)}) \Big]
    +
    C \cdot
    \int_{t_n}^{t}
        \sup_{t_n \leq r \leq s}
        \E \Big[
            V_{\varepsilon,q}(\widetilde Z_r^{(n)})
        \Big]
    \, \dd s
    +
    C h.
\end{aligned}
\end{equation}
By the Gronwall lemma {\color{black}\cite[Theorem 8.1]{mao2007stochastic}} and $\eqref{2023AS-eq:rough_bound}$, we have 
\begin{equation}
    \sup _{t_{n}\leq t \leq t_{n+1}}
    \E \Big [
	V_{\varepsilon,q}(\widetilde Z_t^{(n)})
    \Big ] 
    \leq 
    \E \Big[ 
        V_{\varepsilon,q}(\widetilde Z_{t_n}^{(n)})
    \Big] \e ^{Ch}
    +
    C h \e^{Ch},
\end{equation}
which directly infers that
\begin{equation}\label{2023AS-eq:Yn_before_iteration}
    \E \Big [
	V_{\varepsilon,q}(\widetilde Z_{t_{n+1}}^{(n)})
    \Big ] 
    =
    \E \Big[
        \big(
            \varepsilon
            +
            \vert Y_{n+1} - c_{-1} h Y_{n+1}^{-1} \vert^2
        \big)^{q}
    \Big]
    \leq
    \E \Big[
        \big(
            \varepsilon
            +
            \vert Y_n \vert^{2}
        \big)^q
    \Big]
    \e ^{C h}
    +
    C h \e^{C h}.
\end{equation}
From now on, we set $\varepsilon = 1 + c_{-1} + 2 c_{-1} h$.
On the one hand, one knows that
\begin{equation}
\begin{aligned}
    \E \Big[
        \big(
            \varepsilon
            +
            \vert Y_{n+1} - c_{-1} h Y_{n+1}^{-1} \vert^2
        \big)^{q}
    \Big]
    & =
    \E \Big[
        \big(
            1 + c_{-1} + 2 c_{-1} h
            +
            \vert Y_{n+1} \vert^2
            - 
            2 c_{-1} h 
            +
            {\color{black}c_{-1}^2 h^2}\vert Y_{n+1} \vert^{-2}
        \big)^{q}
    \Big]\\
    & \geq
    \E \Big[
        \big(
            1 + c_{-1}
            +
            \vert Y_{n+1} \vert^2
        \big)^{q}
    \Big].
\end{aligned}
\end{equation}
On the other hand, {\color{black}
by noting
\begin{equation*}
    1 + c_{-1} + 2 c_{-1} h + \vert Y_n \vert^2
    \leq
    1 + 2 h
    +
    c_{-1} (1 + 2 h)
    +
    (1 + 2 h)\vert Y_n \vert^2
    =
    (1 + 2 h) \big(1 + c_{-1} + \vert Y_n \vert^2 \big),
\end{equation*}
one can deduce
}
\begin{equation}
\begin{aligned}
    \E\Big[
    \big(
        \varepsilon + \vert Y_n \vert^2
    \big)^q
    \Big]
    & =
    \E\Big[
    \big(
        1 + c_{-1} + 2 c_{-1} h + \vert Y_n \vert^2
    \big)^q
    \Big]\\
    & \leq
    (1 + 2 h)^q
    \E\Big[
    \big(
        1 + c_{-1} + \vert Y_n \vert^2
    \big)^q
    \Big]\\
    & \leq
    \e^{2 h q }
    \E\Big[
    \big(
        1 + c_{-1} + \vert Y_n \vert^2
    \big)^q
    \Big].
\end{aligned}
\end{equation}
Thus, we conclude that
\begin{equation}
    \E \Big[
        \big(
            1 + c_{-1}
            +
            \vert Y_{n+1} \vert^2
        \big)^{q}
    \Big]
    \leq
    \E\Big[
    \big(
        1 + c_{-1} + \vert Y_n \vert^2
    \big)^q
    \Big]
    \e^{C  h}
    +
    C h e^{C h}.
\end{equation}
By iteration, we arrive at
\begin{equation}
\begin{aligned}
    \sup_{1 \leq n \leq N}
    \E \Big[ \big( 1 + c_{-1} + \vert Y_{n} \vert^2 )^{q} \Big]
    & \leq
    \sup_{1 \leq n \leq N}
    \E \Big[ \big( 1 + c_{-1} + \vert Y_0 \vert^{2} \big)^q \Big]
    \e^{C n h}
    +
    C
    <
     \infty,
\end{aligned}
\end{equation}
as required.
\qed

{\color{black}
Equipped with the assumptions on the modifications and the integrability of numerical solutions, we can derive an important result on the bounds of modification errors.

\begin{cor}
\label{2023AS-cor:Yn_modification_error}
Let Assumptions \ref{2023AS-ass:f_g_properties}, \ref{2023AS-ass:monotone_coupling_condition},
and \ref{2023AS-ass:modification_bounds}
hold.
Then it holds that
\begin{align}
    \E \Big[
        \big\vert f(Y_n) - f_h(Y_n) \big\vert^2
    \Big]
    \leq C h,
    \quad
    \E \Big[
        \big\vert g(Y_n) - g_h(Y_n) \big\vert^2
    \Big]
    \leq C h.
\end{align}

\end{cor}
}
The above properties of the numerical 
approximations play an important role 
in the following error analysis.

\section{Mean-square convergence rate of the proposed scheme } \label{2023AS-section:convergence_rate}
In this section, we present the analysis of the 
mean-square convergence rate for the proposed scheme on the mesh grid.
Two cases are considered including the non-critical case $\kappa + 1 > 2 \rho$ and the general critical case $\kappa + 1 = 2 \rho$. For both cases, we successfully identify 
the desired mean-square convergence rate of order one-half
for the considered scheme. The main convergence result is formulated as follows.

{
\color{black}
\begin{thm}
\label{2023AS-thm:convergence_rate_k+1>=2rho}
Let $\{X_{t}\}_{t\in[0,T]}$ and $\{Y_{n}\}_{0\leq n \leq N}$ be defined by \eqref{2023AS-eq:model_SDE} and \eqref{2023AS-eq:EUPE_scheme},
respectively.
Let Assumptions \ref{2023AS-ass:f_g_properties}, \ref{2023AS-ass:monotone_coupling_condition}, and \ref{2023AS-ass:modification_bounds} hold, with additionally
one of the following conditions standing:
\begin{enumerate}[(1)]
    \item $\kappa + 1 > 2 \rho$;
    \item $\kappa + 1 = 2 \rho$, 
    $m_1,m_2 \leq \tfrac{2 c_2}{c_3^2} + 1$,
    $\tfrac{c_2}{c_3^2} > 2 \kappa - \tfrac32$.
\end{enumerate}
Then there exists a constant $C>0$ independent of $h>0$ such that
\begin{equation}
    \sup _{0\leq n\leq N}
    \E \big[
        \big\vert X_{t_n}-Y_{n}\big\vert^2
    \big]
    \leq 
    C h.
\end{equation}
\end{thm}
}
\textbf{Proof:}
To carry out the error analysis, we first obtain
the error recurrence formula.
By recalling definitions of $X_t$, $Y_n$, $f$, $g$ in \eqref{2023AS-eq:model_SDE}, \eqref{2023AS-eq:EUPE_scheme}, \eqref{2023AS-eq:f_g_original_definition} and denoting 
\[
e_n := X_{t_n} - Y_n,
\quad
n =0, 1,..., N
\]
for brevity, one can  straightforwardly infer
{\color{black}
\begin{equation}\label{2023AS-eq:en_definition}
\begin{aligned}
    e_{n+1}
    -
    c_{-1} \big( X_{t_{n+1}}^{-1} - Y_{n+1}^{-1} \big) h
    & =
    (1 + c_1 h)e_n
    +
    \big( f(X_{t_n}) - f(Y_n) \big) h
    +
    \big( 
        g(X_{t_n}) - g(Y_n)
    \big)
    \Delta W_n
    + 
    R_{n+1},
\end{aligned} 
\end{equation}
where
\begin{equation}\label{2023AS-eq:Rn_definition}
\begin{aligned}
    R_{n+1}
    & :=
    \int_{t_n}^{t_{n+1}}
    \Big(
        c_{-1} X_s^{-1}
        + 
        c_1 X_s
        +
        f(X_s)
    \Big)\, \dd s
    +
    \int_{t_n}^{t_{n+1}}
    g(X_s)
    \, \dd W_s
    -
    c_1 X_{t_n} h
    -
    c_{-1} X_{t_{n+1}}^{-1} h\\
    & \quad 
    -
    \big( f(X_{t_n}) - f(Y_n) \big) h
    -
    \big( g(X_{t_n}) - g(Y_n) \big)
    \Delta W_n
    -
    f_h (Y_n) h
    -
    g_h (Y_n) \Delta W_n\\
    & =
    \underbrace{
    \int_{t_n}^{t_{n+1}}
    \Big[
    c_{-1} \big( X_s^{-1} - X_{t_{n+1}}^{-1} \big)
    +
    c_1 ( X_s - X_{t_n} )
    +
    \big( f(X_s) - f(X_{t_n}) \big)
    +
    ( f(Y_n) - f_h(Y_n) )
    \Big]
    \, \dd s
    }_{=: R_{n+1}^{(1)}}
    \\
    & \quad 
    +
    \underbrace{
    \int_{t_n}^{t_{n+1}}
    \Big[
    \big( 
        g(X_s) - g(X_{t_n})
    \big)
    +
    \big( 
        g(Y_n) - g_h(Y_n)
    \big)
    \Big]
    \, \dd W_s
    }_{=: R_{n+1}^{(2)}}.
\end{aligned}
\end{equation}
}
{\color{black}
Squaring both sides of \eqref{2023AS-eq:en_definition}  and using the Young inequality four times yield
\begin{equation}\label{2023AS-eq:en_squared}
\begin{aligned}
    & \vert e_{n+1} \vert^2
    -
    2 c_{-1} h \big\langle e_{n+1}, X_{t_{n+1}}^{-1} - Y_{n+1}^{-1} \big\rangle
    +
    c_{-1}^2 h^2 
    \big\vert X_{t_{n+1}}^{-1} - Y_{n+1}^{-1} \big\vert^2\\
    & =
    (1 + c_1 h)^2 \vert e_n \vert ^2
    +
    \vert f(X_{t_n}) - f(Y_n) \vert^2 h^2
    +
    \vert g(X_{t_n}) - g(Y_n) \vert^2 |\Delta W_n|^2
    +
    \vert R_{n+1} \vert^2\\
    & \quad
    +
    2 h (1 + c_1 h) \langle e_n,  f(X_{t_n}) - f(Y_n) \rangle
    +
    2 (1 + c_1 h) \langle e_n, ( g(X_{t_n}) - g(Y_n) ) \Delta W_n \rangle\\
    & \quad
    +
    2 (1 + c_1 h) \langle e_n, R_{n+1} \rangle
    +
    2 h \big\langle f(X_{t_n}) - f(Y_n), ( g(X_{t_n}) - g(Y_n) ) \Delta W_n \big\rangle\\
    & \quad
    +
    2 h \big\langle  f(X_{t_n}) - f(Y_n), R_{n+1} \big\rangle
    +
    2 \big\langle ( g(X_{t_n}) - g(Y_n) ) \Delta W_n, R_{n+1} \big\rangle\\
    & \leq
    (1 + c_1 h)^2 \vert e_n \vert ^2
    +
    \vert f(X_{t_n}) - f(Y_n) \vert^2 h^2
    +
    \vert g(X_{t_n}) - g(Y_n) \vert^2 |\Delta W_n|^2
    +
    \vert R_{n+1} \vert^2\\
    & \quad
    +
    2 h \langle e_n,  f(X_{t_n}) - f(Y_n) \rangle
    + 
    c_1 h \vert e_n \vert^2
    +
    c_1 h^3 \vert f(X_{t_n}) - f(Y_n) \vert^2\\
    & \quad
    +
    2 (1 + c_1 h) \langle e_n, ( g(X_{t_n}) - g(Y_n) ) \Delta W_n \rangle
    +
    2 (1 + c_1 h) \langle e_n, R_{n+1}^{(2)} \rangle
    +
    h (1 + c_1 h)^2 \vert e_n \vert^2
    +
    \tfrac{1}{h} \vert R_{n+1}^{(1)} \vert^2\\
    & \quad
    +
    2 h \big\langle f(X_{t_n}) - f(Y_n), ( g(X_{t_n}) - g(Y_n) ) \Delta W_n \big\rangle
    +
    h^2 \vert  f(X_{t_n}) - f(Y_n) \vert^2
    +
    \vert R_{n+1} \vert^2\\
    & \quad
    +
    \delta \vert g(X_{t_n}) - g(Y_n) \vert ^2 |\Delta W_n|^2
    +
    \tfrac{1}{\delta} \vert R_{n+1} \vert^2,
\end{aligned}
\end{equation}
where $\delta > 0$ is some positive constant.
By noting that 
{\color{black}$e_n$, $f(X_{t_n})-f(Y_n)$, $g(X_{t_n})-g(Y_n)$ are $\mathcal{F}_{t_n}$-measurable,
one can easily see that}
\begin{align*}
    \E \big[ \langle e_n, ( g(X_{t_n}) - g(Y_n) ) \Delta W_n \rangle \big]
    & = 0,\\
    \E \big[ \big\langle f(X_{t_n}) - f(Y_n), ( g(X_{t_n}) - g(Y_n) ) \Delta W_n \big\rangle \big]
    & = 0.
\end{align*}
{\color{black}
Since 
$e_n$ is $\mathcal{F}_{t_n}$-measurable,
the property of stochastic integrals implies that
\begin{align*}
    \E \big[ \langle e_n, R_{n+1}^{(2)} \rangle \big]
    =
    \E \bigg[
    \int_{t_n}^{t_{n+1}}
    e_n
    \Big(
    \big( 
        g(X_s) - g(X_{t_n})
    \big)
    +
    \big( 
        g(Y_n) - g_h(Y_n)
    \big)
    \Big)
    \, \dd W_s
    \bigg]
    =
    0.
\end{align*}
}
{\color{black}In addition, we note that}
\begin{align}
    \big\langle e_{n+1}, X_{t_{n+1}}^{-1} - Y_{n+1}^{-1} \big\rangle
    =
    \int_0^1
       -\Big( Y_{n+1} + \theta(X_{t_{n+1}} - Y_{n+1}) \Big)^{-2}
    \dd \theta
    \cdot
    (X_{t_{n+1}} - Y_{n+1})^2
    \leq
    0.
\end{align}
Therefore, we take expectations on both sides of \eqref{2023AS-eq:en_squared} 
and do some rearrangements
to arrive at 
\begin{equation}\label{2023AS-eq:en_expectation}
\begin{aligned}
    \E \big[ \vert e_{n+1} \vert^2 \big]
    & \leq
    (1 + (1 + 3 c_1) h + (c_1^2 + 2 c_1) h^2 + c_1^2 h^3) 
    \E \big[ \vert e_n \vert ^2 \big]
    +
    (2 + c_1 h) h^2 \E \big[ \vert f(X_{t_n}) - f(Y_n) \vert^2 \big] 
    \\
    & \quad
    +
    (1 + \delta) h \E \big[ \vert g(X_{t_n}) - g(Y_n) \vert^2 \big]
    +
    2 h \E \big[ \langle e_n,  f(X_{t_n}) - f(Y_n) \rangle \big]
    \\
    & \quad
    +
    (2 + \tfrac{1}{\delta}) \E \big[ \vert R_{n+1} \vert^2 \big]
    +
    \tfrac{1}{h} \E \big[ \vert R_{n+1}^{(1)} \vert^2 \big].
\end{aligned}
\end{equation}
It is not difficult to check that
\begin{equation}\label{2023AS-eq:1+delta_g+f}
\begin{aligned}
    & (1 + \delta) \big\vert g(X_{t_n}) - g(Y_n) \big\vert^2 
    +
    2  \langle e_n, f(X_{t_n}) - f(Y_n) \rangle\\
    & =
    (1+\delta) c_3^2 (X_{t_n}^{\rho} - Y_n^{\rho})^2 
    - 
    2 c_2 (X_{t_n} - Y_n) (X_{t_n}^{\kappa} - Y_n^{\kappa})
    \\
    & =
    (1+\delta) c_3^2 (X_{t_n} - Y_n)^2 
    \bigg(
        \int_0^1 \rho (Y_n + \theta (X_{t_n} - Y_n))^{\rho - 1}
        \, \dd \theta
    \bigg)^2
    \\
    & \quad
    -
    2 c_2 (X_{t_n} - Y_n)^2
    \int_0^1 \kappa (Y_n + \theta (X_{t_n} - Y_n))^{\kappa - 1}
    \, \dd \theta
\\
& 
\leq
(X_{t_n} - Y_n)^2 
        \int_0^1 
        \big[
        (1+\delta) c_3^2
        \rho^2 
        (
        \theta X_{t_n} +  (1 -\theta )  Y_n
        )^{ 2 \rho - 2}
        -
        2 c_2 \kappa 
        (
        \theta X_{t_n} +  (1 -\theta )  Y_n
        )^{\kappa - 1}
        \big]
        \, \dd \theta
    .
\end{aligned}    
\end{equation}
{\color{black}
Here, we claim that for any $u \geq 0$,
\begin{equation}\label{2023AS-eq:u_bound}
    (1+\delta) c_3^2
    \rho^2 
    u^{ 2 \rho - 2}
    -
    2 c_2 \kappa 
    u^{\kappa - 1}
    \leq
    \tilde{L}.
\end{equation}}
Indeed, for the case  $\kappa + 1 > 2 \rho$, the assertion is obvious by noting $\kappa - 1 > 2 \rho - 2$ and 
{\color{black}the fact that the coefficient $- 2 c_2 \kappa$ of the higher-order term $u^{\kappa - 1}$ is negative.}
For the critical case $\kappa + 1 = 2 \rho$, the assumption $\tfrac{c_2}{c_3^2} > 2 \kappa - \tfrac32$ implies {\color{black}that $ \tfrac{c_2}{c_3^2} > \tfrac{\kappa}{2} + \tfrac32(\kappa - 1) > \tfrac{\kappa}{2}$ and thus $\tfrac{2 c_2 \kappa}{c_3^2 \rho^2} - 1 > \tfrac{\kappa^2}{\rho} - 1 > \rho - 1 > 0$.}
As a result, one can choose 
$
    0 < \delta \leq \tfrac{2 c_2 \kappa}{c_3^2 \rho^2} - 1
$ 
such that
$(1+\delta)c_3^2\rho^2 - 2 c_2 \kappa \leq 0$,
which ensures that {\color{black}the coefficient of $u^{\kappa-1} = u^{2 \rho - 2}$ is still negative.
A combination of the two cases confirms \eqref{2023AS-eq:u_bound}.
Bearing \eqref{2023AS-eq:u_bound} in mind, one can readily derive from \eqref{2023AS-eq:1+delta_g+f} that}
\begin{equation}
\begin{aligned} \label{eq:claim}
    (1+\delta) h \E \big[ \big\vert g(X_{t_n}) - g(Y_n) \big\vert^2 \big]
    +
    2 h \E \big[  \langle e_n, f(X_{t_n}) - f(Y_n) \rangle \big]
    \leq
    \tilde{L}
    {\color{black}h \E \big[ \vert X_{t_n} - Y_n \vert^2 \big] }.
\end{aligned}    
\end{equation}
Taking \eqref{eq:claim} into account, 
one can deduce from \eqref{2023AS-eq:en_expectation} that 
\begin{equation}\label{2023AS-eq:en_delta_g^2+f}
\begin{aligned}
    \E \big[ \vert e_{n+1} \vert^2 \big]
     & \leq
    (1 + (1 + 3 c_1) h + (c_1^2 + 2 c_1) h^2 + c_1^2 h^3 {\color{black} + \tilde L}) 
    \E \big[ \vert e_n \vert ^2 \big]
    +
    (2 + c_1 h) h^2 \E \big[ \vert f(X_{t_n}) - f(Y_n) \vert^2 \big]
    \\
    & 
    \quad
    +
    (2 + \tfrac{1}{\delta}) \E \big[ \vert R_{n+1} \vert^2 \big]
    +
    \tfrac{1}{h} \E \big[ \vert R_{n+1}^{(1)} \vert^2 \big].
\end{aligned}
\end{equation}
Before proceeding further, two cases $\kappa + 1 > 2 \rho$ and $\kappa + 1 = 2 \rho$ are
further considered separately as follows.
%
\begin{enumerate}[{Case} 1:]
\item $\kappa + 1 > 2 \rho$.
    {\color{black}Noticing that $\gamma$ is an arbitrary large constant ensuring $\kappa \leq \gamma + \half$ in the non-critical case due to Assumption \ref{2023AS-ass:monotone_coupling_condition}, with the aid of}
    Lemma \ref{2023AS-lem:X_integrability_k+1>2rho} and Lemma \ref{2023AS-lem:Y_integrability_gamma}, one can infer that
\begin{equation}\label{2023AS-eq:fx-fy_gx-gy_bounds}
    \E \big[ \big\vert
        f(X_{t_n}) - f(Y_n)
    \big\vert^2 \big]
    \leq
    C
    \quad
    \text{and}
    \quad
    \E \big[ \big\vert
        g(X_{t_n}) - g(Y_n)
    \big\vert^2 \big]
    \leq
    C.
\end{equation}
Corollary \ref{2023AS-cor:Yn_modification_error} implies that
\begin{align}\label{2023AS-eq:Yn_modification_error}
    \E \big[
        \big\vert f(Y_n) - f_h(Y_n) \big\vert^2
    \big]
    \leq C h,
    \quad
    \E \big[
        \big\vert g(Y_n) - g_h(Y_n) \big\vert^2
    \big]
    \leq C h.
\end{align}
Equipped with \eqref{2023AS-eq:Yn_modification_error} as well as {\color{black}\eqref{2023AS-eq:Rn_definition},} Lemmas \ref{2023AS-lem:Holder_continuous_k+1>2rho} and \ref{2023AS-lem:f_L^2_bounds} one can directly 
{\color{black}deduce by the Jensen inequality and the H\"{o}lder inequality that
\begin{align*}
    \E \Big[ \big\vert R_{n+1}^{(1)} \big\vert^2 \Big]
    & \leq
    4 \E \bigg[ \Big\vert \int_{t_n}^{t_{n+1}}
        c_{-1} (X_s^{-1} - X_{t_{n+1}}^{-1})
    \dd s \Big\vert^2 \bigg]
    +
    4 \E \bigg[ \Big\vert \int_{t_n}^{t_{n+1}}
        c_1 (X_s- X_{t_{n}})
    \dd s \Big\vert^2 \bigg]\\
    & \quad +
    4 \E \bigg[ \Big\vert \int_{t_n}^{t_{n+1}}
        (f(X_s) - f(X_{t_{n}}))
    \dd s \Big\vert^2 \bigg]
    +
    4 \E \bigg[ \Big\vert \int_{t_n}^{t_{n+1}}
        (f(Y_n) - f_h(Y_n))
    \dd s \Big\vert^2 \bigg]\\
    & \leq
    4 h 
    \int_{t_n}^{t_{n+1}}
    \E \Big[
    \big\vert
        c_{-1} (X_s^{-1} - X_{t_{n+1}}^{-1})
    \big\vert^2
    \Big]
    \dd s 
    +
    4 h 
    \int_{t_n}^{t_{n+1}}
    \E \Big[
    \big\vert
        c_1 (X_s- X_{t_{n}})
    \big\vert^2
    \Big]
    \dd s\\
    & \quad +
    4 h 
    \int_{t_n}^{t_{n+1}}
    \E \Big[
    \big\vert
        (f(X_s) - f(X_{t_{n}}))
    \big\vert^2
    \Big]
    \dd s 
    +
    4 h 
    \int_{t_n}^{t_{n+1}}
    \E \Big[
    \big\vert
        (f(Y_n) - f_h(Y_n))
    \big\vert^2
    \Big]
    \dd s\\
    & \leq
    C h^3.
\end{align*}
In a similar way, but with the It\^{o} isometry leads to}
\begin{equation}\label{2023AS-eq:Rn_bounds_k+1>2rho}
    \E \big[ \big\vert R_{n+1}^{(2)} \big\vert^2 \big]
    \leq 
    C h^2.
\end{equation}
    
    \item $\kappa + 1 = 2 \rho.$

The difference for the critical case comes from the extra restriction that $\gamma \leq \tfrac{c_2}{c_3^2}$, as required by Assumption \ref{2023AS-ass:monotone_coupling_condition}, which together with Lemmas \ref{2023AS-lem:X_integrability_k+1=2rho} and \ref{2023AS-lem:Y_integrability_gamma} ensures the moment boundedness of 
$\E\big[ \vert X_t \vert^{p} \big]$ and $\E\big[ \vert Y_n \vert^{p} \big]$
for $p \leq 2\gamma + 1 \leq \tfrac{2c_2}{c_3^2} + 1$.
Therefore, on the condition $ \tfrac{c_2}{c_3^2} > \kappa - \half$, one can derive $2 \kappa$-th moment bounds for both analytic and numerical solutions
\begin{equation}\label{2023AS-eq:fx-fy_gx-gy_bounds_k+1=2rho}
    \E \big[ \big\vert
        f(X_{t_n}) - f(Y_n)
    \big\vert^2 \big]
    \leq
    C
    \quad
    \text{and}
    \quad
    \E \big[ \big\vert
        g(X_{t_n}) - g(Y_n)
    \big\vert^2 \big]
    \leq
    C.
\end{equation}

Recalling \eqref{2023AS-eq:modification_bounds}, one knows that
\begin{equation}
            \big\vert f (Y_n) - f_h(Y_n) \big\vert^2 
            \leq C h \vert Y_n \vert^{m_1},
            \quad
            \big\vert g (Y_n) - g_h(Y_n) \big\vert^2
            \leq C h \vert Y_n \vert^{m_2},
        \end{equation}
where $m_1, m_2 \leq 2 \gamma + 1 \leq \tfrac{2 c_2}{c_3^2} + 1$. 
Thus for $m_1, m_2 \leq \tfrac{2 c_2}{c_3^2} + 1$,
\begin{align}\label{2023AS-eq:Yn_modification_error_k+1=2rho}
    \E \big[
        \big\vert f(Y_n) - f_h(Y_n) \big\vert^2
    \big]
    \leq C h,
    \quad
    \E \big[
        \big\vert g(Y_n) - g_h(Y_n) \big\vert^2
    \big]
    \leq C h,
\end{align}
by Corollary \ref{2023AS-cor:Yn_modification_error}.
This together with Lemmas \ref{2023AS-lem:Holder_continuous_k+1=2rho} and \ref{2023AS-lem:f_L^2_bounds} {\color{black}similarly} helps us to deduce that,
on the more strict assumption $\tfrac{c_2}{c_3^2} > 2 \kappa - \tfrac32$,
\begin{equation}\label{2023AS-eq:Rn_bounds_k+1=2rho}
    \E \big[ \big\vert R_{n+1}^{(1)} \big\vert^2 \big]
    \leq 
    C h^3,
    \quad
    \E \big[ \big\vert R_{n+1}^{(2)} \big\vert^2 \big]
    \leq 
    C h^2.
\end{equation}
%
\end{enumerate}
Based on the above discussion,  one can derive from \eqref{2023AS-eq:en_delta_g^2+f} that
\begin{equation}
\begin{aligned}
    \E \big[ \vert e_{n+1} \vert^2 \big]
    & \leq
    (1 + C h) \E \big[ \vert e_n \vert ^2 \big]
    +
    C h^2,
\end{aligned}
\end{equation}
which, by iteration gives
\begin{equation}
    \sup _{0\leq n\leq N}
    \E \big[ \vert e_{n} \vert^2 \big]
    \leq
    C h,
\end{equation}
as required.
\qed

Theorem \ref{2023AS-thm:convergence_rate_k+1>=2rho} can be applied to derive convergence rates for any scheme satisfying Assumptions \ref{2023AS-ass:f_g_properties}, \ref{2023AS-ass:monotone_coupling_condition} and \ref{2023AS-ass:modification_bounds}. 
The following corollary shows an application on the schemes in Example \ref{2023AS:eg-tamed_Euler_method}.
\begin{cor}[Convergence rates of semi-implicit tamed Euler methods]
    Let $\{X_{t}\}_{0 \leq t \leq T}$ and $\{Y_{n}\}_{0\leq n \leq N}$ be defined as \eqref{2023AS-eq:model_SDE} and \eqref{2023AS-eq:EUPE_scheme}, respectively.
    Let $f_h$ and $g_h$ be defined as in Example \ref{2023AS:eg-tamed_Euler_method}.
    Suppose that Assumptions \ref{2023AS-ass:f_g_properties}, \ref{2023AS-ass:monotone_coupling_condition}, and \ref{2023AS-ass:modification_bounds} hold, with additionally one of the following conditions standing:
    \begin{enumerate}[(1)]
        \item $\kappa + 1 > 2 \rho$.
        \item $\kappa + 1 = 2 \rho$, 
        $\tfrac{c_2}{c_3^2} \geq (2 \alpha + 1) \kappa - \tfrac12$.
    \end{enumerate}
    Then, there exists a constant $C>0$ independent of the step-size $h$ such that
\begin{equation}
    \sup _{0\leq n\leq N}
    \E \big[
        \big\vert X_{t_n}-Y_{n}\big\vert^2
    \big]
    \leq 
    C h.
\end{equation}
\end{cor}
We would like to mention that, for the general critical case $\kappa + 1 = 2 \rho$, one can observe a slightly more strict restriction on the model parameters 
than that for the 
implicit methods (STMs) in \cite{wang2020mean},
where $\tfrac{c_2}{c_3^2} > 2 \kappa - \tfrac32$ is required. 
This is caused by the need of $m_1, m_2 \leq 2 \gamma + 1$ in order to confirm Assumption \ref{2023AS-ass:modification_bounds}.
More precisely,
as indicated by Example \ref{2023AS:eg-tamed_Euler_method},
for the general critical case, $\gamma \leq \tfrac{c_2}{c_3^2}$ as well as $ 2{\color{black} (2 \alpha\kappa + \rho) } = m_2 < m_1 = 2 (2 \alpha + 1) \kappa \leq 2\gamma + 1$ are required, leading to $ (2 \alpha + 1) \kappa \leq \tfrac{c_2}{c_3^2} + \half$ as mentioned.
}

\section{Numerical experiments} \label{2023AS-section:numerical_experiments}

In this section, numerical experiments are presented to illustrate the previous theoretical results.
Let us consider the following A\"{i}t-Sahalia model
\begin{equation}\label{2023AS-numerical_experiments:AS_model}
\begin{aligned}
		\mathrm{d} X_t=& (c_{-1} X_t^{-1} - c_{0}+ c_{1} X_t - c_{2} X_t^{\kappa})
		\, \mathrm{d} t
		+c_{3} X_t^{\rho} \mathrm{d} W_t,
  \quad
  X_{0}=1>0.
\end{aligned}
\end{equation}
The semi-implicit tamed Euler-type method (TEM) in Example \ref{2023AS:eg-tamed_Euler_method} with $\alpha = \half$ applied to the model reads:
\begin{equation}
    Y_{n+1}
    = 
    Y_n + c_{-1} Y_{n+1}^{-1} h 
    +
    \big( -c_0 + c_1 Y_n - \tfrac{c_2 Y_n^{\kappa}}{1 + \sqrt{h}Y_n^{\kappa}} \big)h
    +
    \tfrac{c_3 Y_n^{\rho}}{1 + \sqrt{h}Y_n^{\kappa}}
    \Delta W_n.
\end{equation}
The following three sets of parameters are taken as numerical tests.
\begin{example}\label{2023AS-eg1} 
    $\kappa=5,\rho=1.5,c_{-1}=1.5,c_{0}=2,c_{1}=1,c_{2}=2,c_{3}=1$.
\end{example}

\begin{example}\label{2023AS-eg2} 
    ${\color{black}\kappa=3,\rho=2},c_{-1}=1.5,c_{0}=2,c_{1}=1,c_{2}=4,{\color{black}c_{3}=0.5}$.
\end{example}
\begin{example}\label{2023AS-eg3} 
{\color{black}
$\kappa=2,\rho=1.5,c_{-1}=2,c_{0}=3,c_{1}=4,c_{2}=7,c_{3}=\sqrt{2}$.
}
\end{example}

Note that Example \ref{2023AS-eg1} satisfies the non-critical condition $\kappa+1>2\rho$, Example \ref{2023AS-eg2} satisfies the  critical condition $\kappa+1=2\rho$, {\color{black}and Example  \ref{2023AS-eg3}, borrowed from \cite[SET II]{halidias2022boundary}, satisfies the critical condition with $\kappa = 2$, $\rho = 1.5$ and $\tfrac{c_2}{c_3^2} = 2 \kappa - \half$.}
%
%
The strong convergence rates are tested with the endpoint $T=1$.
{\color{black}For comparison, the backward Euler method (BEM) \cite{szpruch2011numerical}, 
which has been proved to be strong convergent with order 0.5 for the model \cite{wang2020mean}, is also tested. Moreover,}
numerical approximations {\color{black}produced by the BEM} with the step-size $h_{exact}=2^{-14}$ is used to simulate the ``exact" solutions, while $h= 2^{-i},i=4, 5, ..., 9$ are used for numerical solutions over $10^4$ Brownian paths.
The estimation error is calculated in terms of $ e_h: = \bigg( \sup\limits_{0\leq n\leq N} \E\Big[  \vert X_{t_n} - Y_n \vert^2 \Big] \bigg)^{\half}$. 

Fig. \ref{2023AS:fig-convergence_rate}
{\color{black}and Fig. \ref{2023AS:fig-convergence_rate_BPES}} plot the average sample errors against various step-sizes on a log-log scale {\color{black}for Examples \ref{2023AS-eg1}, \ref{2023AS-eg2}, and \ref{2023AS-eg3}}, with BEM in a blue solid line, TEM in a red solid line, {\color{black}and the boundary preserving explicit scheme (BPES) proposed in \cite{halidias2022boundary} in a green solid line. 
The expected order $0.5$ of strong convergence is easily detected for both BEM and TEM}.
By a least square fitting, the resulting convergence rate $q$ and the least squares residual are shown in Table \ref{2023AS-tab:least_squares_fit}, where one can also confirm the predicted convergence rate.
{\color{black}
For  Example \ref{2023AS-eg3} of the particular critical case $\kappa = 2, \rho = 1.5$, the BPES 
shows order one convergence and outperforms 
the other two order one-half schemes in terms of accuracy. Nevertheless, the BPES is simply designed for the particular critical case $\kappa = 2, \rho = 1.5$ (see \cite{halidias2022boundary} for more details), but the proposed explicit scheme, TEM, works for all cases.}

In addition, the computational costs of BEM and the proposed TEM scheme over $10^4$ Brownian paths are shown in Table \ref{2023AS-tab:time_costs}, and 50 out of the $10^4$ paths are shown in Fig. \ref{2023AS-fig:sample_path}. 
It is easy to see that the TEM can 
preserve positivity and reduce computational costs significantly, compared with BEM, especially as the parameter $\kappa$ increases.
\begin{figure}[htp]
\includegraphics[width=0.5\linewidth, height=0.3\textheight]{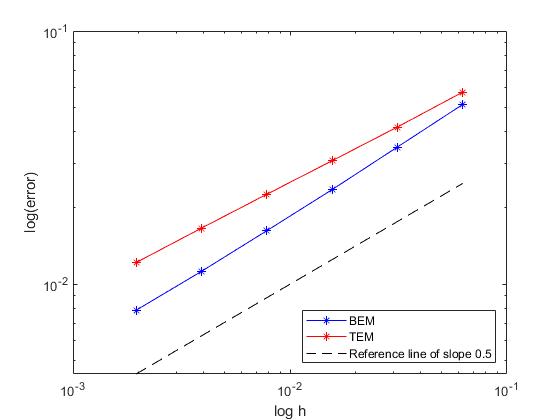}
\includegraphics[width=0.5\linewidth, height=0.3\textheight]{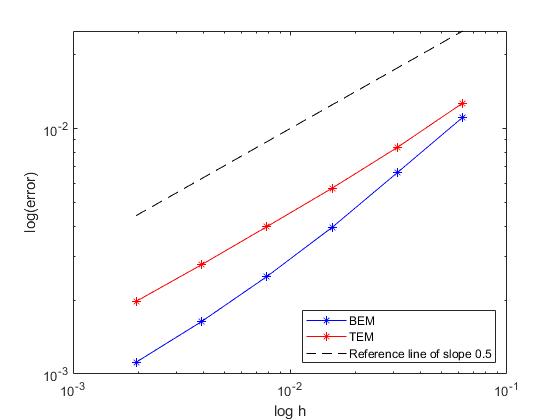}
\caption{Strong convergence rate of {\color{black}TEM and BEM} for Example \ref{2023AS-eg1} (left) and Example \ref{2023AS-eg2} (right).\label{2023AS:fig-convergence_rate}}
\end{figure}

\begin{figure}[htp]
\centering
\includegraphics[width=0.5\linewidth, height=0.3\textheight]{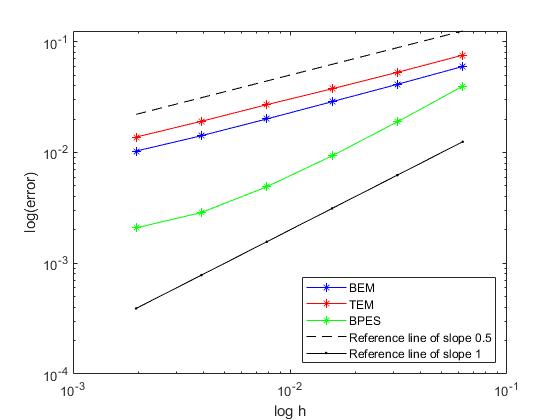}
\caption{{\color{black}Strong convergence rate of TEM and BEM for Example \ref{2023AS-eg3}.}\label{2023AS:fig-convergence_rate_BPES}}
\end{figure}

\begin{figure}[htp]
\centering
\includegraphics[width=0.6\linewidth, height=0.4\textheight]{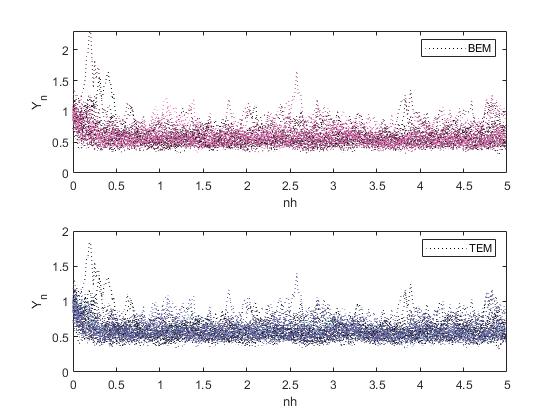}
\caption{Fifty paths of numerical solutions for BEM and TEM}
\label{2023AS-fig:sample_path}
\end{figure}

\begin{table}[htp]
	\centering
	\setlength{\tabcolsep}{7mm}
        \caption{A least square fit for the convergence rate $q$}\label{2023AS-tab:least_squares_fit}
	\begin{tabular}{c c c}
		\toprule[2pt]
		&  BEM	& TEM \\
		\midrule 
		{\color{black}Eg. \ref{2023AS-eg1}} & $q=0.5409$, $resid=0.0269$ &{\color{black} $q=0.4458$, $resid=0.0074$} \\
		{\color{black}Eg. \ref{2023AS-eg2}} & {\color{black} $q=0.6658$, $resid=0.1161$} & {\color{black}$q=0.5347$, $resid=0.0524$}\\
            {\color{black}Eg. \ref{2023AS-eg3}} & {\color{black}$q=0.5095$, $resid=0.0352$} & {\color{black}$q=0.4906$, $resid=0.0174$}\\
		\bottomrule [2pt]
	\end{tabular}
	\vspace{2pt}
\end{table}

\begin{table}[htp]
	\centering
	\setlength{\tabcolsep}{20mm}
        \caption{Time cost (seconds) for BEM and TEM over $10^4$ Brownian paths}\label{2023AS-tab:time_costs}
	\begin{tabular}{c c c}
		\toprule[2pt]
		  &  BEM & TEM\\
		\midrule 
		  Example \ref{2023AS-eg1} & {\color{black}135.586} & {\color{black}51.905} \\
            Example \ref{2023AS-eg2} & {\color{black}78.977} & {\color{black}36.189} \\
            {\color{black}Example \ref{2023AS-eg3}} & {\color{black}34.085} & {\color{black}25.358} \\
		\bottomrule [2pt]
	\end{tabular}
	\vspace{2pt}
\end{table}

\section{Conclusion} \label{2023AS-section:conclusion}

In this paper, we propose and analyze an explicit time-stepping scheme for the strong approximations of the generalized A\"{i}t-Sahalia model. 
Based on a semi-implicit strategy and 
proper modifications, the numerical scheme is designed to be positivity preserving and easily implementable with cheap computational costs 
compared with existing positivity preserving schemes.
The mean-square strong convergence rate of the scheme is carefully analyzed and proved to be order $0.5$. 
Numerical experiments are also provided to support the theoretical findings.
As an ongoing project, we successfully construct higher order efficient positivity preserving time-stepping schemes for the A\"{i}t-Sahalia model \cite{jiang2023unconditionally} and try to cover a larger range of stochastic models \cite{neuenkirch2014first}.

\section*{Declarations}

\subsection*{Ethical Approval} 
Not applicable

\subsection*{Availability of supporting data}
The authors confirm that the data supporting the findings of this study are available within the article.

\subsection*{Competing interests}
The authors declare no competing interests.

\subsection*{Funding}
This work is supported by the National Natural Science Foundation of China (Grant Nos. 12071488, 12371417, 11971488), the Natural Science Foundation of Hunan Province (Grant No.2020JJ2040).

\subsection*{Authors' contributions} 
All the authors contribute equally to this work.

\subsection*{Acknowledgments} 
Not applicable

\bibliographystyle{abbrv}
\bibliography{reference.bib}
\end{document}